\documentclass[sn-mathphys-num]{sn-jnl}

\usepackage{multirow}%
\usepackage{amsthm}%
\usepackage{mathrsfs}%
\usepackage[title]{appendix}%
\usepackage{textcomp}%
\usepackage{manyfoot}%
\usepackage{booktabs}%
\usepackage{algorithm}%
\usepackage{algorithmicx}%
\usepackage{algpseudocode}%
\usepackage{listings}%

\usepackage{amsfonts,amssymb,amsmath,mathtools,cancel,eucal} 
\usepackage{nicefrac} 
\DeclareMathAlphabet{\bbol}{U}{bbold}{m}{n} 
\usepackage{upgreek} 
\usepackage{multicol} 

\usepackage{graphicx,subfig,tikz} 

\usetikzlibrary{matrix,arrows.meta} 
\newcommand{\freccia}{-{Stealth[scale=1.5]}} 

\usepackage{color} 

\usepackage{comment}

\newif \ifdetailshide  

\detailshidetrue   

\ifdetailshide
    \providecommand{\details}[1]{}
\else
    \providecommand{\details}[2]{$~$ \newline \noindent{\bf Details:} 
    #1 $~$ \hfill $\blacksquare$ \newline }
\fi

\newtheorem{thm}{Theorem}[section]
\newtheorem{prop}[thm]{Proposition}

\newtheorem{cor}[thm]{Corollary}
\theoremstyle{definition}
\newtheorem{dfn}[thm]{Definition}
\newtheorem{oss}[thm]{Remark}
\newtheorem{ass}[thm]{Assumption}

\newcommand{\defeq}{\coloneqq}

\newcommand{\NN}{\mathbb{N}}
\newcommand{\RR}{\mathbb{R}}
\renewcommand{\SS}{\mathbb{S}}

\newcommand{\per}[1]{\mathring{#1}}
\newcommand{\ZZ}{\mathbb{Z}}

\newcommand{\bs}{\boldsymbol}
\newcommand{\bi}{{\bs i}}
\newcommand{\be}{{\bs e}}
\newcommand{\bn}{{\bs n}}
\newcommand{\bv}{{\bs v}}
\newcommand{\bw}{{\bs w}}
\newcommand{\bx}{{\bs x}}
\newcommand{\bB}{{\bs B}}
\newcommand{\bE}{{\bs E}}
\newcommand{\bF}{{\bs F}}
\newcommand{\bG}{{\bs G}}

\newcommand{\bJ}{{\bs J}}
\newcommand{\bT}{{\bs T}}
\newcommand{\hT}{{\hat T}}
\newcommand{\hbT}{{\hat \bT}}
\newcommand{\hbE}{{\hat{\bs E}}}
\newcommand{\hB}{{\hat{B}}}
\newcommand{\arr}[1]{{\bs{\mathsf{#1}}}}
\renewcommand{\pmb}[1]{\bs{#1}}

\newcommand\arrsigma{{\bs{\mathsf{\sigma}}}}
\newcommand\ttc{{\mathtt{c}}}
\newcommand\tte{{\mathtt{e}}}
\newcommand\ttn{{\mathtt{n}}}
\newcommand{\cB}{\mathcal{B}}

\newcommand{\cF}{\mathcal{F}}
\newcommand{\cI}{\mathcal{I}}

\newcommand{\cO}{\mathcal{O}}

\newcommand{\mat}[1]{\bbol{#1}}
\newcommand{\tmat}[1]{\tilde{\mat{#1}}}

\def\rmd{\, \mathrm{d}} 

\DeclareMathOperator{\bgrad}{{\bf grad}}
\DeclareMathOperator{\hbgrad}{\widehat{\bf grad}}
\DeclareMathOperator{\grad}{{grad}}
\DeclareMathOperator{\curl}{{curl}}
\DeclareMathOperator{\hcurl}{\widehat{curl}}
\let\div\relax 
\DeclareMathOperator{\div}{{div}}
\DeclareMathOperator{\hdiv}{\widehat{div}}
\DeclareMathOperator{\bcurl}{{\bf curl}}
\DeclareMathOperator{\hbcurl}{\widehat{\bf curl}}
\DeclareMathOperator{\Span}{Span}

\providecommand{\abs}[1]{\lvert#1\rvert}

\providecommand{\norm}[1]{\lVert#1\rVert}

\providecommand{\sprod}[2]{\langle#1,#2\rangle}
\providecommand{\ssprod}[2]{\langle\!\langle#1,#2\rangle\!\rangle}
\def\bbb{\phantom{\big|}}

\begin{document}

\title[Broken-FEEC framework for tensor-product splines in polar domains]{A broken-FEEC framework for structure-preserving discretizations in polar domains with tensor-product splines}


\author[1]{\fnm{Yaman} \sur{Güçlü}}\email{yaman.guclu@ipp.mpg.de}
\equalcont{All three authors contributed equally to this work.}

\author[2]{\fnm{Francesco} \sur{Patrizi}}\email{francesco.patrizi@unifi.it}
\equalcont{All three authors contributed equally to this work.}

\author*[1]{\fnm{Martin} \sur{Campos Pinto}}\email{martin.campos-pinto@ipp.mpg.de}
\equalcont{All three authors contributed equally to this work.}

\affil[1]{\orgname{Max-Planck-Institut f{\"u}r Plasmaphysik}, \orgaddress{\street{Boltzmannstra{\ss}e 2}, \city{Garching bei M{\"u}nchen}, \postcode{85748}, \country{Germany}}}


\affil[2]{\orgdiv{Department of Mathematics}, \orgname{University of Rome Tor Vergata}, \orgaddress{\street{Via della Ricerca Scientifica 1}, \postcode{00133}, \city{Rome}, \country{Italy}}}


\abstract{We propose a novel projection-based approach to 
derive structure-preserving Finite Element Exterior Calculus (FEEC) 
discretizations using standard tensor-product splines 
on domains with a polar singularity. 
This approach follows the main lines of broken-FEEC schemes 
which define stable and structure-preserving 
operators in non-conforming discretizations of the de Rham complex. 
Here, we devise a polar broken-FEEC framework that enables 
the use of standard tensor-product spline spaces
while ensuring stability and smoothness for the solutions, 
as well as the preservation of the de Rham structure.
A benefit of this approach is the ability to reuse codes that implement standard splines on smooth parametric domains, 
and efficient solvers such as Kronecker-product spline interpolation.
Our construction is based on two pillars: the first one is an explicit characterization of smooth polar spline spaces within the tensor-product splines ones, which are either discontinuous or non square-integrable as a result of the singular polar pushforward operators.
The second pillar consists of local, explicit and matrix-free conforming projection operators that map general tensor-product splines onto 
smooth polar splines, and that commute with the differential operators of the de Rham complex.
}




\maketitle


\tableofcontents

\section{Introduction}\label{sec1}

In computational electromagnetism, finite element discretizations that preserve the de~Rham structure play a fundamental role in developing stable and accurate numerical schemes \cite{bossavit_computational_1998,Hiptmair.2002.anum,arnold_finite_2006}.
These discretizations also ensure the conservation of essential 
physical properties such as the Hamiltonian structure of 
Vlasov-Maxwell equations~\cite{kraus_gempic_2017,campos_pinto_variational_2022}, 
while naturally accommodating curvilinear geometries 
through the use of parametric mappings and pushforward operators
\cite{Hiptmair.2002.anum,perse_geometric_2021}.
One particularly effective class of methods employs tensor-product 
splines, which can be assembled into stable sequences
of discrete differential forms \cite{buffa_isogeometric_2010,buffa_isogeometric_2011}.

When dealing with polar parametrizations, which are 
especially convenient for disk- or torus-shaped domains, 
a well-known issue arises: because of the singularity exhibited 
by the mapping at the pole
the pushforward tensor-product splines have poor smoothness 
or integrability properties, which prevents their use 
in practical computations. 

This issue has been previously 
addressed in the literature, notably by Toshniwal and co-authors
who proposed a construction of smooth {\em polar splines} in 
\cite{deepesh1}. By defining proper extraction matrices which
correspond to the coefficients of the smooth polar basis splines in the full 
(a priori singular) tensor-product basis, this construction 
results in splines which are smooth on a surface
parametrized by a spline polar mapping, and 
in \cite{deepesh2} it has been 
extended to smooth rational splines on spheres and ellipsoids.
The construction of polar splines for the different spaces of
a de Rham complex was then proposed in \cite{deepesh3} for 
polar spline surfaces, and in \cite{patrizi_isogeometric_2025} for 
toroidal spline domains. In \cite{holderied_magneto-hydrodynamic_2022}, commuting projection operators have been proposed for axisymmetric toroidal domains with a polar spline cross section, along the lines of the finite element exterior calculus (FEEC) methodology \cite{arnold_finite_2006, arnold_finite_2010}.
Overall, these works provide reduced spline spaces that allow for stable, structure-preserving approximations of de~Rham complexes
on domains with polar singularities.

In this article, we propose an alternative approach which allows to design structure-preserving schemes with standard tensor-product splines, despite their singularity on the polar domain.
Our method is inspired by the {\em broken-FEEC} approach \cite{campos_pinto_broken-feec_2025}, where finite element spaces with good mathematical properties are seen as subspaces of simpler spaces with less constraints.
This projection-based approach was first introduced in the CONGA (conforming/non-conforming Galerkin) scheme for Maxwell's equations \cite{campos_pinto_gauss-compatible_2016}, in order to represent the electric field in a discontinuous function space associated with block-diagonal mass matrices and local $L^2$ projections. Later it was extended to discontinuous multipatch discretizations of full de~Rham complexes~\cite{guclu_broken_2023,campos_pinto_broken-feec_2025}, where the full decoupling of the patchwise function spaces results in Galerkin Hodge operators (defined as weighted $L^2$ projections) which are local, as well as the dual differential operators and the dual commuting projections.

Here, the broken-FEEC method is applied to the case of a single-patch 2D domain~$\Omega$ with a polar singularity, which presents additional challenges due to the fact that the associated pushforward operators do not map in general bounded functions to square integrable ones. In particular, the non-conforming spaces defined as the pushforward of the tensor-product spline spaces in logical variables are not, in general, subspaces of $L^2(\Omega)$. Moreover, the singular pushforward operators do not always commute with the differential operators of the de Rham complex.

Another difficulty is that the conforming polar spline subspaces do not have a straightforward characterization in terms of tensor-product B-spline coefficients. In this work, we consider two de Rham complexes associated with different types of smoothness: the first one is a low regularity complex corresponding to the usual Hilbert de Rham complex (characterized by $L^2(\Omega)$ fields with $L^2(\Omega)$ differentials), while the second one is a higher regularity complex where scalar fields in the first space of the complex are also required to be $C^1(\Omega)$, while vector fields from the second space are also required to be $C^0(\Omega)$. 
%
Our results are as follows:
\begin{itemize}
   \item 
   We provide a \textit{full characterization} of both 
   spline subcomplexes using linear relations between tensor-product B-spline coefficients.
   As $H^1$ regularity for the scalar potential space is equivalent to $C^0$ regularity,
   we refer to the two subcomplexes as the $C^0$ and $C^1$ complexes respectively,
   using the regularity of their first space;
   \item 
   For both complexes, we propose simple explicit \textit{conforming projections} that map the singular tensor-product splines onto the conforming one;
   \item 
   We show that the polar pushforward operators satisfy a \textit{regular} commutation property with the gradient operator but only a \textit{singular} one with the curl, that involves a Dirac distribution at the pole;
   \item 
   Finally, we prove that our conforming projections can be composed with standard tensor-product spline interpolation and histopolation operators to yield new \textit{commuting projections} onto the conforming subcomplexes.
\end{itemize}

As a result, our approach makes it possible to reuse standard finite element codes based on tensor-product B-splines when solving problems on polar domains. Instead of re-implementing the discrete Hodge, projection or differential operators in a new basis, it is enough to implement the proposed conforming projection operators using the explicit formulas provided in this article, and to compose them with the standard operators associated with the tensor-product B-splines.
In particular, the commuting property of our conforming projections ensure that the resulting discretization maintains the desired structural properties. An interesting feature of the resulting commuting projections is that they are based on tensor-product interpolation and histopolation, which involve the fast inversion of Kronecker products of univariate collocation matrices.



The remainder of the article is organized as follows. In Section~\ref{sec:conga} we recall the main steps of the projection-based
broken-FEEC method for structure-preserving discretizations with non-conforming finite element spaces,
and we further specify in which sense tensor-product splines are not conforming after a pushforward on a domain with a polar singularity.
In Section~\ref{sec:conforming} we present our main results, namely:
(i)~a~characterization of the conforming spline de Rham subcomplexes
corresponding to $H^1$ or $C^1$ regularity for the scalar potential space,
the former case being equivalent with $C^0$ regularity,
(ii)~sequences of explicit conforming projection operators mapping tensor-product splines onto the $C^0$ and $C^1$ subcomplexes, and
(iii)~new sequences of commuting projection operators onto these subcomplexes, obtained by composing our discrete conforming projections with standard interpolation and histopolation operators in the tensor-product spline spaces.
The characterization of the smooth spline subcomplexes is proven in Section~\ref{sec:charconf},
and after studying the integrability and establishing the regular and singular commuting properties of the pullback and pushforward operators in Section~\ref{sec:pbpf}, we prove some key commuting properties of our conforming projections in Section~\ref{sec:commutation}, as well as those of the tensor-product geometric projections on the polar domain.
Finally, we illustrate our construction with several
numerical experiments in Section~\ref{sec:num}.
Our results show that isogeometric polar broken-FEEC approximations
of Poisson and Maxwell's equations using the proposed projection 
matrices yield stable and high-order accurate solutions.

\section{Principles of broken-FEEC schemes on domains with a polar singularity} 
\label{sec:conga}

\subsection{Conforming discretizations of de Rham complex in $\RR^2$}

Structure-preserving finite element discretizations of two-dimensional domains
usually involve discrete de Rham complexes of the form
\begin{equation} \label{dR_c}
	\begin{tikzpicture}[ampersand replacement=\&, baseline] 
	\matrix (m) [matrix of math nodes,row sep=3em,column sep=3em,minimum width=2em] {
      V^{0}_h ~ \bbb
         \& ~~ V^{1}_h ~ \bbb
            \& ~~ V^{2}_h ~ \bbb
		\\
	};
	\path[-stealth]
	(m-1-1) edge node [above] {$\bgrad$} (m-1-2)
	(m-1-2) edge node [above] {$\curl$} (m-1-3)
	;
	\end{tikzpicture}
\end{equation}
where the finite-dimensional spaces $V^\ell_h$ are {\em conforming}, 
i.e., subspaces of the natural spaces $V^\ell$ involved in the 
underlying continuous problem. A common choice is to define 
the latter as the Hilbert de Rham complex \cite{arnold_finite_2006}:
\begin{equation} \label{dR_V_1}
	\begin{tikzpicture}[ampersand replacement=\&, baseline] 
	\matrix (m) [matrix of math nodes,row sep=3em,column sep=3em,minimum width=2em] {
      V^{0} \defeq H^1(\Omega) ~ \bbb
         \& ~~ V^{1} \defeq H(\curl;\Omega) ~ \bbb
            \& ~~ V^{2} \defeq L^2(\Omega). ~ \bbb
		\\
	};
	\path[-stealth]
	(m-1-1) edge node [above] {$\bgrad$} (m-1-2)
	(m-1-2) edge node [above] {$\curl$} (m-1-3)
	;
	\end{tikzpicture}
\end{equation}

We note that in $\RR^2$ there exists another de Rham complex, namely 
\begin{equation} \label{dR_cD} 
	\begin{tikzpicture}[ampersand replacement=\&, baseline] 
      \matrix (m) [matrix of math nodes,row sep=3em,column sep=3em,minimum width=2em] {
         \tilde V^{0} \defeq H(\bcurl;\Omega) ~ \bbb
            \& ~~ \tilde V^{1} \defeq H(\div;\Omega) ~ \bbb
               \& ~~ \tilde V^{2} \defeq L^2(\Omega), ~ \bbb
         \\
      };
      \path[-stealth]
      (m-1-1) edge node [above] {$\bcurl$} (m-1-2)
      (m-1-2) edge node [above] {$\div$} (m-1-3)
      ;
      \end{tikzpicture}
\end{equation}
which can be related to the grad-curl complex by a standard
rotation argument: indeed, we have
\begin{equation} \label{startdiffeq}
   \bcurl \phi = R\bgrad \phi 
      \quad \text{ and } \quad 
   \div \bv = \curl R^{-1} \bv 
      \quad \text{ with} \quad 
   R = \begin{pmatrix} 0 & 1 \\ -1 & 0  \end{pmatrix}.
\end{equation}

In this article we will consider discretizations of the 
grad-curl complex \eqref{dR_V_1}, but we note that that all our results
readily apply to to the div-curl complex \eqref{dR_cD} through a rotation
of the vector fields.

\subsection{Projection approach for non-conforming FEEC spaces}

In practice, it may be convenient to work with 
{\em non-conforming} spaces
$$
W^{\ell}_h \not\subset V^\ell 
\quad \text{ such that } \quad
V^{\ell}_h = W^{\ell}_h \cap V^\ell
$$
such as fully discontinuous spaces (which have block-diagonal mass matrices).
The broken-FEEC approach then relies on discrete {\em conforming projections}
\begin{equation}
   \label{P_conf}
   P^\ell: W^{\ell}_h \to V^{\ell}_h
\end{equation}
which yield a new, non-conforming de Rham complex
\begin{equation} \label{dR_nc}
	\begin{tikzpicture}[ampersand replacement=\&, baseline] 
	\matrix (m) [matrix of math nodes,row sep=3em,column sep=3em,minimum width=2em] {
      W^{0}_h ~ \bbb
         \& ~~ W^{1,}_h ~ \bbb
            \& ~~ W^{2,}_h ~ \bbb.
		\\
	};
	\path[-stealth]
	(m-1-1) edge node [above] {$\bgrad P^0$} (m-1-2)
	(m-1-2) edge node [above] {$\curl P^1$} (m-1-3)
	;
	\end{tikzpicture}
\end{equation}
This framework allows to derive stable numerical 
approximations in so-called broken-FEEC spaces for 
many problems in electromagnetism,
as described for instance in \cite{guclu_broken_2023}.
A simple example is the Poisson equation, 
\begin{equation}\label{eq:PoissonEq}
\left\{
\begin{array}{rl}
-\Delta \phi = f & \quad \text{in } \Omega \\
\phi = 0 & \quad \text{on }\partial\Omega.
\end{array}
\right.
\end{equation}
By adding homogeneous boundary conditions to the definition 
of the spaces \eqref{dR_V_1}, 
and hence to the conforming projections \eqref{P_conf},
one can approximate the solution to \eqref{eq:PoissonEq}
by a non-conforming potential $\phi_h \in W^{0}_h$ satisfying
\begin{equation}\label{eq:PoissonConga}
\sprod{\bgrad P^0 \phi_h}{\bgrad P^0 \psi_h} 
+ \alpha \sprod{(I - P^0) \phi_h}{(I - P^0) \psi_h} = \sprod{f}{P^0 \psi_h} 
\end{equation}
for all $\psi_h \in W^{0}_h$.
Similarly, the electric field solving Maxwell's time harmonic equation
\begin{equation}\label{eq:Maxwell}
   \left\{
   \begin{array}{rl}
   (-\omega^2 + \curl \curl) \bE = \bJ & \quad \text{in }{\Omega}\\
   \bn \times \bE = 0 & \quad \text{on }\partial\Omega,
   \end{array}
   \right.
\end{equation}
can be approximated by a non-conforming discrete field $\bE_h \in W^{1}_h$ solution to 
\begin{equation}\label{eq:MaxwellConga}
   -\omega^2 \sprod{P^1 \bE_h}{P^1 \bF_h}
   + \sprod{\curl P^1 \bE_h}{\curl P^1 \bF_h} 
+ \alpha \sprod{(I - P^1) \bE_h}{(I - P^1) \bF_h} = \sprod{\bJ}{P^1 \bF_h} 
\end{equation}
for all $\bF_h \in W^{1}_h$.
For both problems above, one shows that 
for any stabilization parameter $\alpha > 0$ the 
non-conforming solutions actually belong to the conforming
spaces and solve the conforming Galerkin 
problems~\cite{guclu_broken_2023}.

Non-conforming approximations to time-dependent 
problems can also be derived,
for instance Maxwell's equations 
\begin{equation}\label{eq:tMaxwell}
    \left\{
    \begin{array}{rl}
    \partial_t \bE - \curl \bB &= -\bJ 
    \\
    \partial_t \bB + \curl \bE &= 0
    \end{array}
    \right.
\end{equation}
which can be discretized in space as
\begin{equation}\label{eq:tMaxwell_h}
    \left\{
    \begin{array}{rl}
    \sprod{\partial_t \bE_h}{\bF_h} 
    - \sprod{\bB_h}{\curl P^1 \bF_h} &= -\sprod{\bJ}{P^1 \bF_h} 
    \\
    \partial_t \bB_h + \curl P^1 \bE_h &= 0
    \end{array}
    \right.
\end{equation}
for all test function $\bF_h \in W^{1}_h$.
As this discretization has no stabilization mechanism
its solution does not belong to the conforming space, however it is stable and shown to preserve a discrete version of the Gauss law
$\div \bE = \rho$, 
see \cite{campos_pinto_gauss-compatible_2016}. 

\smallskip
\begin{oss}
   \label{rem:mass_unbounded}
   Here we have implicitely assumed that the non-conforming functions 
   where square integrable, i.e., that $W^\ell_h \subset L^2(\Omega)$.
   Although this inclusion holds for non-conforming spaces that result
   from a relaxation of the continuity constraints across regular interfaces
   as studied in \cite{guclu_broken_2023},
   it is no longer the case for the non-conforming spaces considered
   in this article, which result from a singular mapping.
   This issue will be addressed in Section~\ref{sec:regmass} by introducing regularised mass matrices.
\end{oss}

\subsection{Domains with a polar singularity} 

In this article we consider problems posed on 
a polar domain, that is an open domain $\Omega \subset \RR^2$ 
whose closure $\overline \Omega$ is the image of the logical annulus
$\hat \Omega \defeq [0,L] \times (\RR / 2\pi \ZZ)$ by a $C^1$ surjective mapping
\begin{equation}
   \label{F_surj}
   F: \hat \Omega \ni \begin{pmatrix}
      s \\ \theta
   \end{pmatrix} \mapsto \bx \in \overline \Omega
\end{equation}
that collapses the logical edge
$\{0\} \times (\RR / 2\pi \ZZ)$ (simply denoted $\{s=0\}$ hereafter) 
to a single point $\bx_0 \in \Omega$, 
which we will call the pole. 
We further assume that $F$ induces a $C^1$ diffeomorphism 
between the {\em punctured} polar domain and its preimage, namely
\begin{equation} \label{hatOm0-Om0} 
   \overline \Omega_0 \defeq \overline \Omega \setminus \{\bx_0\}
   \qquad \text{ and } \qquad
   \hat \Omega_0 
   \defeq \hat \Omega \setminus \{s=0\}
    = F^{-1}(\overline \Omega_0).
\end{equation}

Here we will consider two types of such polar mappings: 
\begin{itemize}
    \item the ``analytical'' polar mapping
    \begin{equation} \label{F_apol}
    F: \begin{pmatrix}
        s \\ \theta 
    \end{pmatrix}
    \mapsto
    \bx_0 + 
    \begin{pmatrix}
        s  \cos \theta \\ s \sin \theta 
    \end{pmatrix}
    \end{equation}
    corresponding to a disk domain $\Omega = \Omega_{\rm disk}$,
     
    \item a spline mapping of the form
    \begin{equation} \label{F_spol}
      F: \begin{pmatrix}
         s \\ \theta 
      \end{pmatrix}
      \mapsto
      \bx_0 + \sum_{i=0}^{n_s-1}\sum_{j=0}^{n_\theta-1} \pmb{P}_{ij} B_i(s)\per{B}_j(\theta)
   \end{equation}
   with regular B splines along the periodic variable $\theta$ as described below, and regular control points close to the pole, i.e.    
   \begin{equation} \label{F_spol_P01}
         \pmb{P}_{ij} = \begin{pmatrix}
         \rho_i \cos\theta_j \\ 
         \rho_i \sin\theta_j
      \end{pmatrix}
      ~~ \text{ for } ~~
      \left\{ \begin{aligned}
         &i \in \{0, 1\}
         \\ 
         &0 \le j < n_\theta
      \end{aligned}\right.
      ~~  \text{ with } ~~ 
      \left\{ \begin{aligned}
         &0 = \rho_0 < \rho_1
         \\ 
         &\theta_j \defeq j \Delta_\theta, 
         ~ \Delta_\theta \defeq \tfrac{2\pi}{n_\theta}~.
      \end{aligned}\right.
   \end{equation}
\end{itemize}

In \eqref{F_spol}, $B_i$ and $\per{B}_j$ are B-splines of degree 
$p \ge 1$ in the respective $s$ and $\theta$ variables 
(the upper ring denoting a periodic spline).
B-splines along $s$ are defined by an open knot sequence
\begin{equation}
   s_0 = \cdots = s_p < s_{p+1} < \cdots < s_{n_s-1} < s_{n_s} = \cdots = s_{n_s + p} = L,
\end{equation}
and along $\theta$ we use the regular angles $\theta_j = j \Delta_\theta$ 
as knots (extended to all $j \in \ZZ$ by periodicity).
We note that the above assumption of $F$ being a $C^1$ diffeomorphism
between the domains \eqref{hatOm0-Om0} implies
\begin{equation} \label{detJpos}
    \det J_F(s,\theta) > 0 \qquad \text{ for all $s > 0$ and all $\theta$}.
\end{equation}

In addition, we make the following assumption on the spline resolution along $\theta$.
\begin{ass} \label{ass:ntheta}
   In the case of a spline mapping \eqref{F_spol}--\eqref{F_spol_P01},
   the number of (regular) knots $\theta_j$ 
   satisfies 
   \begin{equation} \label{ntheta}
      n_\theta = 4 n' 
      \qquad \text{ with an integer } n' \ge p \ge 1.
   \end{equation}
\end{ass}
This assumption will be useful for studying the strength of the mapping singularity in Proposition \ref{prop:singFpol} below. It will also
allow us to use standard discrete trigonometry relations that are  
reminded in the Appendix, together with some details on B-splines.   

\smallskip 
\begin{oss}    
    Using standard tools \cite{buffa_isogeometric_2011}, 
    our results directly extend to any smooth 
    deformations of the above domains, that is, to any mapping 
    of the form $G \circ F$ where $G$ is a smooth diffeomorphism
    and $F$ is one of the polar mappings above.
\end{oss}

\smallskip 
The singularity of polar mappings can be characterized as follows.

\smallskip 
\begin{dfn} \label{def:singpol}
   We say that a mapping $F$ has a first order polar singularity if the following properties hold:
   \begin{itemize}
      \item[(i)]
      its Jacobian matrix 
      is of the form 
      \begin{equation} \label{JF_prop}
         J_F(s,\theta) 
         =
         \begin{pmatrix}
         C(\theta) + \cO(s) & s(C'(\theta) +\cO(s))
         \\ 
         S(\theta) + \cO(s) &  s(S'(\theta) +\cO(s))
         \end{pmatrix},
         \quad \det J_F(s,\theta) = s(D(\theta) + \cO(s))
      \end{equation}
      where $C$ and $S$ are $C^1$ functions, $D(\theta) = (CS'-SC')(\theta) > 0$ holds for all $\theta$. Here we use the classical notation
      $\cO(s)$ to denote a generic function 
      (depending a priori on $s$ and $\theta$, and which value may change at each occurence) satisfying 
      \begin{equation}
         \label{Os}
         \abs{\cO(s)} \le c s
      \end{equation}
      for some constant $c > 0$ that may depend on the mapping parameters.
   
      \item[(ii)] the bound $s D_* \le \det J_F(s,\theta)$ 
      holds for a positive constant $D_* > 0$, and 
      the inverse Jacobian determinant is of the form
      \begin{equation}
         \label{detJ_inv}
         \det J_F^{-1}(s,\theta) = \frac{1+\cO(s)}{s D(\theta)},
      \end{equation}
      again for a generic function $\cO(s)$ satisfying \eqref{Os}.

   \end{itemize}
\end{dfn}

\medskip
\begin{prop} \label{prop:singFpol}
   Both the analytical mapping \eqref{F_apol} and the generic spline mapping \eqref{F_spol} (assuming \eqref{detJpos}) 
   have a first order polar singularity in the sense of Definition~\ref{def:singpol}.
\end{prop}
\smallskip

\begin{proof}    
    The Jacobian matrix reads
    \begin{equation} \label{JF}
        J_F(s,\theta) 
        = \begin{pmatrix}
        \partial_s F_1(s,\theta) & \partial_\theta F_1(s,\theta) 
        \\ 
        \partial_s F_2(s,\theta) & \partial_\theta F_2(s,\theta)
        \end{pmatrix}.
    \end{equation}
    In the case of the analytical 
    mapping \eqref{F_apol} we thus have
    \begin{equation} \label{JF_apol}
    J_F(s,\theta) 
    =
    \begin{pmatrix}
    \cos \theta & - s \sin \theta
    \\ 
    \sin \theta &  s \cos \theta 
    \end{pmatrix},
    \qquad \det J_F(s,\theta) = s,
    \end{equation}    
    which yields \eqref{JF_prop} with $C(\theta) = \cos \theta$,
    $S(\theta) = \sin \theta$, $D(\theta) = 1$ and $\cO(s) = 0$.

    In the case of the spline mapping \eqref{F_spol}, 
    we apply the derivative formula \eqref{Bprime} from the Appendix to write
    \begin{equation} \label{pole_JF}
        \left\{
        \begin{aligned}
          \partial_s F_1(s,\theta) 
             &= \sum_{i=0}^{n_s-2} (\rho_{i+1}-\rho_i) M_i(s)
             \sum_{j=0}^{n_\theta-1} \cos\theta_j \per{B}_j(\theta) 
          \\
          \partial_\theta F_1(s,\theta)
             &= \sum_{i=0}^{n_s-1} \rho_i B_i(s)
             \Big(\sum_{j=0}^{n_\theta-1} \cos\theta_j \per{B}_j\Big)'(\theta) 
          \\
          \partial_s F_2(s,\theta) 
             &= \sum_{i=0}^{n_s-2} (\rho_{i+1}-\rho_i) M_i(s)
             \sum_{j=0}^{n_\theta-1} \sin\theta_j \per{B}_j(\theta)
          \\
          \partial_\theta F_2(s,\theta) 
             &= \sum_{i=0}^{n_s-1}\rho_i B_i(s)
             \Big(\sum_{j=0}^{n_\theta-1} \sin\theta_j \per{B}_j\Big)'(\theta).
        \end{aligned}
        \right.
    \end{equation}    
    Using the relations \eqref{BM_0}--\eqref{BM_0p} 
    (also from the Appendix)
    satisfied at the pole, and the fact that $\rho_0 = 0$, 
   we further see that
    \begin{equation*} 
            \sum_{i=0}^{n_s-2} (\rho_{i+1}-\rho_i) M_i(s)
             = \rho_1 M_0(0) + \cO(s)
          \quad \text{ and } \quad
            \sum_{i=0}^{n_s-1} \rho_i B_i(s)
             = s (\rho_1 M_0(0) + \cO(s))
    \end{equation*}
    hold with a function $\cO(s)$ satisfying \eqref{Os}. 
    This yields \eqref{JF_prop} 
    with 
    \begin{equation}
      \label{CS_spol}
      C(\theta) 
      = \rho_1 M_0(0) \sum_{j=0}^{n_\theta-1} \cos\theta_j \per{B}_j(\theta)
      \quad \text{ and } \quad
      S(\theta) 
      = \rho_1 M_0(0) \sum_{j=0}^{n_\theta-1} \sin\theta_j \per{B}_j(\theta).
   \end{equation}
   To show that $D(\theta) = (CS' - SC')(\theta)$ is positive for all $\theta$,
    we compute
    \begin{equation*}
        \begin{split}
        \frac{D(\theta)}{(\rho_1M_0(0))^2} 
        &= 
        \sum_{j,k=0}^{n_\theta-1} \Big(
            \cos \theta_j (\sin \theta_{k+1} - \sin \theta_k) 
            \per{B}_j\per{M}_k
        -  (\cos \theta_{j+1} - \cos \theta_j) \sin \theta_k 
        \per{M}_j\per{B}_k 
        \Big) (\theta)
        \\
        &=
        \sum_{j,k=0}^{n_\theta-1} \big(
            (\cos\theta_j\sin\theta_{k+1}-\sin\theta_j\cos\theta_{k+1}) 
            -(\cos\theta_j\sin\theta_k-\cos\theta_k\sin\theta_j) \big)
            \mathring{B}_j\mathring{M}_k(\theta) 
        \\ 
        &=
        \sum_{j,k=0}^{n_\theta-1} 
            (\sin\theta_{k-j+1}-\sin\theta_{k-j})
            \mathring{B}_j\mathring{M}_k(\theta)  
        \\ 
        &= \sum_{j=0}^{n_\theta-1} \sum_{k=-p}^{p-1} (\sin\theta_{k+1}-\sin\theta_k)\per{B}_j \per{M}_{j+k}(\theta)
        \end{split}
    \end{equation*}
    where in the last equality we have used \eqref{perBM} to reduce the second sum.
    By the mean value theorem, for all
    $k \in \{-p, \ldots, p-1\}$ there exists 
    $\bar{\theta}_k \in (\theta_k, \theta_{k+1}) 
    \subset (-\theta_{p}, \theta_{p})$ 
    such that
    $\sin\theta_{k+1} - \sin\theta_k 
    = \Delta\theta\cos\bar{\theta}_k $
    and from the inequality $n_\theta \ge 4p$ derived from 
    \eqref{ntheta}, we infer that 
    $\theta_{p} = \frac{2\pi p}{n_\theta} \le \frac{\pi}{2}$.
    In particular, we have 
    $\abs{\bar{\theta}_k} < \frac{\pi}{2}$ for all $k$, hence
    $$
    \sin\theta_{k+1} - \sin\theta_k = \Delta\theta\cos\bar{\theta}_k > 0
    \quad \text{ for } \quad
    k \in \{-p, \ldots, p-1\}.
    $$
    The bound
    $D(\theta) > 0$ follows from the non-negativity of the B and M basis splines.
    Observe that if $n_\theta > 4 p$, we have
    $
    \sin\theta_{k+1} - \sin\theta_k = \Delta\theta\cos\bar{\theta}_k
    > \Delta \theta\cos \theta_p,
    $ hence, using \eqref{pup} and \eqref{pup-M},
    $$
    \frac{D(\theta)}{(\rho_1 M_0(0))^2} \ge \sum_{j,k} \Delta \theta \cos \theta_p \per{B}_j(\theta) \per{M}_{j+k}(\theta)
    \ge \cos \theta_p > 0 
    $$
    which gives a uniform bound with respect to $n_\theta$
    (note that $\cos \theta_p \to 1$ as $n_\theta \to \infty$).    
    The bounds on $\det J_F$, as well as the form \eqref{detJ_inv},
    follow easily by using this lower bound on $D(\theta)$ and the 
    assumption \eqref{detJpos}.
\end{proof}

\subsection{Singularity of the tensor-product spline de Rham complex}

On the logical domain we consider a de Rham complex of
tensor-product splines \cite{buffa_isogeometric_2011}, of the form
\begin{equation} \label{dR_W}
  \hat W^0_h := \SS_{p,p}(\hat\Omega) \xrightarrow{ \mbox{$~ \hbgrad ~$}}
    \hat W^1_h := \begin{pmatrix} \SS_{p-1,p}(\hat\Omega) \\ \SS_{p,p-1}(\hat\Omega)\end{pmatrix} 
    \xrightarrow{ \mbox{$~ \hcurl ~$}}
      \hat W^2_h := \SS_{p-1,p-1}(\hat\Omega).
\end{equation}
Here the hatted derivatives correspond to the logical variables $(s,\theta) \in \hat\Omega$,
and 
$
\SS_{p_1, p_2}(\hat\Omega) := \SS_{p_1}([0,L]) \otimes \SS_{p_2}(\RR / 2\pi \ZZ)
$
are tensor-product spline spaces 
associated with the same breakpoints as the spline mapping 
\eqref{F_spol}.
These spaces are classically equipped with tensor-product spline bases,
\begin{equation} \label{hT_bases}
   \left\{
   \begin{aligned}
      \text{ for } \hat W^0_h: \quad
         &\hT^0_{ij}(s,\theta) = B_i(s)\per{B}_j(\theta)
      \\
      \text{ for } \hat W^1_h: \quad
         &\hbT_{ij}^s(s,\theta) = \begin{pmatrix}
         M_i(s)\per{B}_j(\theta)\\
         0
         \end{pmatrix}
         ~\text{ and }~
         \hbT_{ij}^\theta(s,\theta) = \begin{pmatrix}
         0\\
         B_i(s)\per{M}_j(\theta)
         \end{pmatrix}
      \\
      \text{ for } \hat W^2_h: \quad
         &\hT^2_{ij}(s,\theta) = M_i(s)\per{M}_j(\theta).
   \end{aligned}
   \right.
\end{equation}
involving the B and M splines defined in the Appendix.
On the ``physical'' polar domain $\Omega$, the standard 
approach~\cite{Hiptmair.2002.anum, buffa_isogeometric_2011, kreeft2011mimetic}
consists in 
pushing forward the logical spaces \eqref{dR_W},
i.e., setting
\begin{equation} \label{Wh}
W^0_h \defeq \cF^0 \hat W^0_h, \qquad W^1_h \defeq \cF^1 \hat W^1_h, \qquad W^2_h \defeq \cF^2 \hat W^2_h
\end{equation}
with pushforward operators defined as
\begin{equation} \label{pf}
  \left\{
  \begin{aligned}
  &\cF^0 : \hat \phi \mapsto \phi := \hat \phi \circ F^{-1}
  \\
  &\cF^1 : \hat \bv \mapsto \bv :=  \big(J_F^{-T} \hat \bv \big)\circ F^{-1}
  \\
  &\cF^2 : \hat f \mapsto f :=  \big(\det J_F^{-1} \hat f \big)\circ F^{-1}.
  \end{aligned}
  \right.
\end{equation}
This amounts to defining the physical spaces as the span of the pushforward tensor-product spline basis functions:
\begin{equation} \label{T_bases}
   \left\{
   \begin{aligned}
      W^0_h &= \Span \Big(%
      \big\{ T^0_{ij} \defeq \cF^0 \hT^0_{ij} : 
      0 \le i < n_s, 0 \le j  < n_\theta \big\}
      \Big)
      \\
      W^1_h &= \Span \Big(%
      \big\{ \bT_{ij}^s \defeq \cF^1 \hbT_{ij}^s : 
      0 \le i < n_s-1, 0 \le j  < n_\theta \big\}
      \\
      & \mspace{100mu}
      \cup 
      \big\{ \bT_{ij}^\theta \defeq \cF^1 \hbT_{ij}^\theta : 
      0 \le i < n_s, 0 \le j  < n_\theta \big\}
      \Big)
      \\
      W^2_h &= \Span \Big(%
      \big\{ T^2_{ij} \defeq \cF^2 \hT^2_{ij} : 
      0 \le i < n_s-1, 0 \le j  < n_\theta \big\}
      \Big).
   \end{aligned}
   \right.
\end{equation}

Since $F: \hat \Omega_0 \to \overline \Omega_0$ is a $C^1$ diffeomorphism, 
\eqref{pf} defines bijective operators between continuous functions away from the pole, that is, 
$$
\cF^\ell: C^0(\hat \Omega_0) \to C^0(\overline \Omega_0) \quad \text{ is a bijection for $\ell \in \{0, 1, 2\}$}
$$ 
with inverse operators $\cB^\ell 
: C^0(\overline \Omega_0) \to C^0(\hat \Omega_0)$ called pullbacks,
\begin{equation} \label{pb}
  \left\{
  \begin{aligned}
  &\cB^0 : \phi \mapsto \hat \phi := \phi \circ F
  \\
  &\cB^1 : \bv \mapsto \hat \bv :=  J_F^{T} (\bv \circ F)
  \\
  &\cB^2 : f \mapsto \hat f :=  \det J_F (f \circ F).
  \end{aligned}
  \right.
\end{equation}
On the full polar domain, i.e. when the pole is included,
one easily verifies that the pullbacks are actually well-defined.

\medskip
\begin{prop} \label{prop:pb_Lq}
   The pullbacks $\cB^\ell$ are continuous mappings from 
   $L^q(\Omega)$ to $L^q(\hat \Omega)$, $1 \le q \le \infty$,
   as well as from $C^0(\overline \Omega)$ to $C^0(\hat \Omega)$.
\end{prop}
\smallskip

\begin{proof}
   Simply use the bounds \eqref{JF_prop} on the mapping Jacobian $J_F$.
\end{proof}

A key property of regular pullbacks and pushforward is the commutation with the 
differential operators from de Rham complex \eqref{dR_V_1}, 
see e.g. \cite{Hiptmair.2002.anum}. 
{\em Away from the pole}, this commutation also holds here 
for general smooth functions: 
for any $\hat \phi$ and $\hat \bv$ in $C^1(\hat \Omega)$, the equalities 
\begin{equation} \label{com_pf}
   \bgrad \cF^0 \hat \phi = \cF^1 \hbgrad \hat \phi 
   \quad \text{ and } \quad
   \curl \cF^1 \hat \bv = \cF^2 \hcurl \hat \bv 
   \quad \text{ hold on } \Omega_0,
\end{equation}
where we remind that $\Omega_0$ is the punctured domain 
\eqref{hatOm0-Om0}, and that the hatted derivatives correspond to the 
logical variables $(s,\theta) \in \hat \Omega$.
We point out that here, the logical functions $\hat \phi$ and $\hat \bv$ 
are assumed to be smooth up to $s = 0$, however  
\eqref{com_pf} does not a priori hold on the full polar domain $\Omega$
because 
of the singularity of the pushforward operators:
First, one may have that $\cF^0 \hat \phi \notin H^1(\Omega)$ or
even that $\cF^1 \hat \bv \notin L^2(\Omega)$ for smooth functions on $\hat \Omega$.
Second, the relations \eqref{com_pf} (and in particular the second one)
do not hold for all smooth functions on $\hat \Omega$.
The latter issue will be analyzed in 
Proposition~\ref{prop:cF_diff}.
The former one can be stated as follows.

\medskip
\begin{prop} \label{prop:Wnc}
    The spaces \eqref{Wh} are not a conforming discretization of
    \eqref{dR_V_1}, in the sense that
    $$
    W^0_h \not\subset H^1(\Omega), 
    \qquad W^1_h \not\subset L^2(\Omega),
    \qquad W^2_h \not\subset L^2(\Omega).
    $$
\end{prop}

\begin{proof}
    The functions in $W^0_h$ are bounded on the bounded domain $\Omega$ so they are $L^2(\Omega)$,
    however their gradient may not be. To verify this, consider 
    a spline 
    that does not vanish on the pole, such as
    $\phi = \cF^0 \hT_{0j}^0$ with an arbitrary $j$, that is
    $\phi(\bx) = B_0(s) \per{B}_j(\theta)$ with $\bx = F(s,\theta)$,
    and use the commutation relation \eqref{com_pf}
    for $\bx \in \Omega_0$, i.e. $s > 0$: with the derivative formulas 
    \eqref{Bprime} and \eqref{perBprime} from the Appendix, this gives
    $$
    \bgrad \phi(\bx) = (\cF^1 \hbgrad \hT_{0j}^0)(\bx) 
        = J_F^{-T} (s, \theta)\begin{pmatrix}
        -M_0(s)\per{B}_j(\theta) \\
        B_0(s)(\per{M}_{j-1}(\theta)-\per{M}_j(\theta))
    \end{pmatrix}
    $$
    with a matrix of the form
    \begin{align*}
      J_F^{-T}(s, \theta) 
      &= \frac{1}{\det J_F}\begin{pmatrix}
         \partial_2 F_2 & -\partial_1 F_2 \\
         -\partial_2 F_1 & \partial_1 F_1
     \end{pmatrix}(s, \theta)
     \\
     &= \frac{1}{s (D(\theta)+\cO(s))}\begin{pmatrix}
         s(S'(\theta) +\cO(s)) & -S(\theta) + \cO(s) \\
         -s(C'(\theta) +\cO(s)) & C(\theta) + \cO(s)
     \end{pmatrix}
   \end{align*}
    where the second equality uses \eqref{JF_prop}. 
    Using Equation~\eqref{detJ_inv} 
    and $B_0(s) = 1 + \cO(s)$,
    this implies that
    $$
    \bgrad \phi(\bx) = \frac{ {\bs G}(\theta) }{s}
    + \bs{g}(\bx)
    \quad \text{ with } \quad
    {\bs G}(\theta) \defeq
    \frac{1}{D(\theta)} \begin{pmatrix}
        - S(\theta)(\per{M}_{j-1}(\theta)-\per{M}_j(\theta))
        \\
        C(\theta)(\per{M}_{j-1}(\theta)-\per{M}_j(\theta))
    \end{pmatrix}
    $$
    and some bounded $\bs{g} \in L^\infty(\Omega)$. 
    This easily shows that $\bgrad \phi \notin L^2(\Omega)$, indeed
    \begin{multline}
      \int_\Omega \abs{\bgrad \phi(\bx) - \bs{g}(\bx)}^2 \rmd \bx 
      = \int_{\hat \Omega} \frac{ \abs{{\bs G}(\theta)}^2 }{s^2} \det J_F(s,\theta)\rmd s \rmd \theta
      \\
      = \int_{\hat \Omega} \frac{ \abs{{\bs G}(\theta)}^2 }{s} (D(\theta) + \cO(s))\rmd s \rmd \theta
      = +\infty.  
    \end{multline}
    The same argument (replacing $\bgrad \phi$
    by $\cF^1 \pmb{T}_{0j}^\theta$) 
    shows that $W^1_h \not\subset L^2(\Omega)$.
    Finally the non-inclusion for the space $W^2_h$ 
    is verified by considering $f = \cF^2 \hT_{0j}^2$,
    which, using again Equation~\eqref{detJ_inv}, 
    writes as
    $$
    f(\bx) 
        = \det J^{-1}_F \hT_{0j}^2(s,\theta) 
        = \frac {M_0(s)\per{M}_j(\theta)}{\det J_F}
        = \frac {G(\theta)}{s} + g(\bx)
    $$
    with $G(\theta) = \frac{M_0(0)\per{M}_j(\theta)}{D(\theta)}$ and 
    a bounded $g \in L^\infty(\Omega)$. One clearly 
    has $f \notin L^2(\Omega)$.
\end{proof}

\section{Construction of a polar broken-FEEC framework with tensor-product splines} 
\label{sec:conforming}

\subsection{Work plan}
As a consequence of Proposition~\ref{prop:Wnc}, we see that 
the tensor-product spaces $W^\ell_h$ are not suitable for 
finite element methods.
They are suitable, however, for applying the projection-based 
broken-FEEC approach described in Section~\ref{sec:conga}. 
To do so we need to specify conforming spline spaces, 
and propose discrete projection operators 
that map arbitrary tensor-product splines onto the conforming subspaces.
In addition, we will require that our conforming projections 
are {\em local} (only coefficients close to the pole should be modified) 
and {\em commuting}: this will be detailed just below.

In this article we will consider two notions of conformity:
the first one corresponds to the simplest (maximal) Hilbert complex \eqref{dR_V_1}, 
that is
\begin{equation} \label{dR_V}
	\begin{tikzpicture}[ampersand replacement=\&, baseline] 
      \matrix (m) [matrix of math nodes,row sep=3em,column sep=3em,minimum width=2em] {
         V^{0} \defeq H^1(\Omega) ~ \bbb
            \& ~~ V^{1} \defeq H(\curl;\Omega) ~ \bbb
               \& ~~ V^{2} \defeq L^2(\Omega) ~ \bbb
         \\
      };
      \path[-stealth]
      (m-1-1) edge node [above] {$\bgrad$} (m-1-2)
      (m-1-2) edge node [above] {$\curl$} (m-1-3)
      ;
      \end{tikzpicture}
\end{equation}
and the second one consists of smoother functions, namely
\begin{equation} \label{dR_U}
	\begin{tikzpicture}[ampersand replacement=\&, baseline] 
	\matrix (m) [matrix of math nodes,row sep=3em,column sep=3em,minimum width=2em] {
      U^{0} \defeq C^1(\Omega) \cap H^1(\Omega) ~ \bbb
         \& ~~ U^{1} \defeq C^0(\Omega) \cap H(\curl;\Omega) ~ \bbb
            \& ~~ U^{2} \defeq L^2(\Omega). ~ \bbb
		\\
	};
	\path[-stealth]
	(m-1-1) edge node [above] {$\bgrad$} (m-1-2)
	(m-1-2) edge node [above] {$\curl$} (m-1-3)
	;
	\end{tikzpicture}
\end{equation}

STOP 

In the sequel we will refer to \eqref{dR_V} and \eqref{dR_U} as the $H^1$ and $C^1$ complexes, respectively.
To each of these complexes, that is for $Z = V$ or $U$, we associate conforming subspaces
\begin{equation} \label{Zh}
   Z^\ell_h \defeq W^\ell_h \cap Z^\ell,
   \quad \text{ $\ell = 0, 1, 2$.} 
\end{equation}
Our \textbf{first task} will be provide an explicit characterization
of these subspaces in terms of their tensor-product spline coefficients,
and our \textbf{second task} will be to design local conforming projections of the form
\begin{equation} \label{PZ}
   P^\ell_Z : W^\ell_h \to Z^\ell_h,
\end{equation}
that is, operators which satisfy 
$P^\ell_Z W^\ell_h = Z^\ell_h$ and 
$(P^\ell_Z)^2 = P^\ell_Z$ and which only modify spline coefficients close to the pole.

A \textbf{third task} will be to build projection operators 
$\Pi^\ell_Z$ onto the discrete spaces $Z^\ell_h$,
for which the following diagram commutes:
\begin{equation} \label{cd}
	\begin{tikzpicture}[ampersand replacement=\&, baseline] 
	\matrix (m) [matrix of math nodes,row sep=3em,column sep=3em,minimum width=2em] {
      ~~ H^1(\Omega) ~ \bbb
         \& ~~ H(\curl;\Omega) ~ \bbb
            \& ~~ L^2(\Omega) ~ \bbb
		\\
         ~~ Z^0_h ~ \bbb
            \& ~~ Z^1_h ~ \bbb
                \& ~~ Z^2_h ~. \bbb
		\\
	};
	\path[-stealth]
	(m-1-1) edge node [above] {$\bgrad$} (m-1-2)
			edge node [right] {$\Pi^0_Z$} (m-2-1)
	(m-1-2) edge node [above] {$\curl$} (m-1-3)
					edge node [right] {$\Pi^1_Z$} (m-2-2)
	(m-1-3) edge node [right] {$\Pi^2_Z$} (m-2-3)
	(m-2-1) edge node [above] {$\bgrad$} (m-2-2)
	(m-2-2) edge node [above] {$\curl$} (m-2-3)
	;
	\end{tikzpicture}
\end{equation}
Commuting projections play indeed an important role in the preservation of certain structure properties at the discrete level: 
they allow to write discrete Hodge-Helmholtz decompositions 
with proper harmonic field spaces of the correct dimension 
\cite{arnold_finite_2006},
or to preserve the Hamiltonian structure in variational particle 
discretization of Vlasov-Maxwell equations \cite{campos_pinto_variational_2022}.
Note that in \eqref{cd} the operators $\Pi^\ell_Z$ may be
unbounded, in the sense that their domains 
$D(\Pi^\ell_Z)$ may be proper subspaces of the spaces from the top row,
typically consisting of smoother functions. Nevertheless, these domains
need to be dense: we therefore need to extend the projections $P^\ell_Z$ 
to infinite-dimensional spaces.

We observe that commuting projection operators on polar spline spaces
have already been proposed in \cite{holderied_magneto-hydrodynamic_2022},
by introducing special sets of geometric degrees of freedom close to the pole.
In this work we will propose an alternative construction which takes
advantage of (and preserves) the tensor-product structure of the underlying spline spaces.

A classical way to obtain commuting projections is indeed to 
interpolate geometric degrees of freedom, as is done by the De Rham map
and described in \cite{Bochev_Hyman_2006_csd,Gerritsma.2011.spec}:
the commuting properties of the resulting projection then 
essentially follows from the Stokes theorem.
On the Cartesian domain $\hat \Omega$, this approach yields 
{\em logical geometric projections} $\hat \Pi^\ell_W$ 
on the tensor-product spline spaces, which are fast to 
evaluate thanks to the Kronecker structure of the collocation matrices
to be inverted. We will review their main properties 
in Section \ref{sec:geoproj}.
Their counterparts on the physical domain, namely 
\begin{equation} \label{Pi_W}
   \Pi^\ell_W \defeq \cF^\ell \hat \Pi^\ell_W \cB^\ell
   : D(\Pi^\ell_W) \to W^\ell_h
\end{equation}
may be referred to as the (tensor-product) {\em polar geometric projections}.

These projections involve the same Kronecker product collocation matrices, however they do not map into the conforming spaces $Z^\ell_h$: 
for instance, an interpolating spline in $W^0_h$ 
has no reason to be $C^1$ at the pole.
Our solution consists in considering the composed 
{\em conforming-geometric projections} 
\begin{eqnarray} \label{Pi_Z}
    \Pi^\ell_Z \defeq P^\ell_Z \Pi^\ell_W : D(\Pi^\ell_W) \to Z^\ell_h
    \qquad \text{ (with $Z = V$ or $U$) }.
\end{eqnarray}
The commutation of the diagram \eqref{cd} then becomes 
an additional requirement in the construction 
of the conforming projections $P^\ell_Z$.

 
In summary, our construction will yield the following diagram
(with $Z = V$ or $U$)
\begin{equation}
   \label{cd_3D}
\begin{tikzpicture}[baseline=(current  bounding  box.center)]
\matrix (m)[matrix of math nodes,column sep=2.5em,row sep=4em]{
 \pgfmatrixnextcell H^1(\Omega) \pgfmatrixnextcell 
 \pgfmatrixnextcell H(\curl;\Omega) \pgfmatrixnextcell 
 \pgfmatrixnextcell L^2(\Omega) 
 \\
Z_h^0 \pgfmatrixnextcell \pgfmatrixnextcell Z_h^1 \pgfmatrixnextcell \pgfmatrixnextcell Z_h^2 \pgfmatrixnextcell 
\\
 \pgfmatrixnextcell W_h^0 \pgfmatrixnextcell \pgfmatrixnextcell W_h^1 \pgfmatrixnextcell \pgfmatrixnextcell W_h^2
 \\};
\draw[\freccia] (m-1-2) --node[above]{$\grad$} (m-1-4);
\draw[\freccia] (m-1-4) --node[above]{$\curl$} (m-1-6);
\draw[\freccia] (m-2-1) --node[above,xshift=-.5cm]{$\grad$} (m-2-3);
\draw[\freccia] (m-2-3) --node[above,xshift=-.5cm]{$\curl$} (m-2-5);
\draw[\freccia] (m-3-2) --node[above,xshift=-.65cm]{$\grad\, P^0_Z$} (m-3-4);
\draw[\freccia] (m-3-4) --node[above,xshift=-.5cm]{$\curl\, P^1_Z$} (m-3-6);

\draw[\freccia] (m-1-2) --node[right,yshift=.5cm]{$\Pi^0_W$} (m-3-2);
\draw[\freccia] (m-1-4) --node[right,yshift=.5cm]{$\Pi^1_W$} (m-3-4);
\draw[\freccia] (m-1-6) --node[right]{$\Pi^2_W$} (m-3-6);

\draw[\freccia] (m-1-2) edge[bend right=15] node[above,xshift=-.1cm]{$\Pi^0_Z$} (m-2-1);
\draw[\freccia] (m-1-4) edge[bend right=15]node[above,xshift=-.1cm]{$\Pi^1_Z$} (m-2-3);
\draw[\freccia] (m-1-6) edge[bend right=15]node[above,xshift=-.1cm]{$\Pi^2_Z$} (m-2-5);
\draw[dashed, \freccia] (m-2-1) edge[bend right=15] node[below,xshift=.1cm]{$I$} (m-1-2);
\draw[dashed,\freccia] (m-2-3) edge[bend right=15] node[below,xshift=.1cm]{$I$} (m-1-4);
\draw[dashed,\freccia] (m-2-5) edge[bend right=15] node[below,xshift=.1cm]{$I$} (m-1-6);

\draw[\freccia] (m-3-2) edge[bend right=15]node[right,yshift=.1cm]{$P^0_Z$} (m-2-1);
\draw[\freccia] (m-3-4) edge[bend right=15]node[right,yshift=.1cm]{$P^1_Z$} (m-2-3);
\draw[\freccia] (m-3-6) edge[bend right=15]node[right,yshift=.1cm]{$P^2_Z$} (m-2-5);
\draw[dashed,\freccia] (m-2-1) edge[bend right=15]node[left,xshift=-.2cm]{$I$} (m-3-2);
\draw[dashed,\freccia] (m-2-3) edge[bend right=15]node[left,xshift=-.2cm]{$I$} (m-3-4);
\draw[dashed,\freccia] (m-2-5) edge[bend right=15]node[left,xshift=-.2cm]{$I$} (m-3-6);
\end{tikzpicture}
\end{equation}
where all the paths commute except for the ones involving the identity operators, 
which are written here to specify space embeddings 
such as $Z^\ell_h \subset W^\ell_h$. 
As above, the projections $\Pi^\ell_W$ and $\Pi^\ell_Z$ should be seen as unbounded operators
with dense domains $D(\Pi^\ell_Z) = D(\Pi^\ell_W)$.

\subsection{Characterization of the conforming spline spaces}

To define proper commuting projections we begin by 
characterizing the conforming spaces \eqref{Zh}
in terms of explicit constraints on the spline coefficients.

\medskip
\begin{thm}[Hilbert polar spline complex] \label{thm:Vseq}
   If the mapping $F$ has a first order polar singularity in the sense of 
   Definition~\ref{def:singpol}, then the conforming spaces 
   \eqref{Zh} for the Hilbert complex \eqref{dR_V}, namely
   \begin{equation} \label{Vh}
      \left\{
      \begin{aligned}            
      V^0_h &\defeq W^0_h \cap H^1(\Omega),
      \\
      V^1_h &\defeq W^1_h \cap H(\curl;\Omega),
      \\
      V^2_h &\defeq W^2_h \cap L^2(\Omega),
      \end{aligned}
      \right.   
   \end{equation}
   are characterized by the relations
    \begin{equation}\label{eq:VlCharacterization}
        \left\{
        \begin{aligned}
        &V^0_h = \big\{\phi = \sum_{i,j} \phi_{ij} T^0_{i,j} \in W^0_h : 
           ~ \phi_{0j} = \gamma_0 
           ~ \text{ for all $j$} 
           \big\},
        \\
        &V^1_h = \Big\{\bv = \sum_{i,j} v^s_{ij} \pmb{T}^s_{i,j} 
              + \sum_{i,j} v^\theta_{ij} \pmb{T}^\theta_{i,j} \in W^1_h : 
              \left\{
               \begin{aligned}
               &v^\theta_{0j} = 0
               \\
               &v^\theta_{1j} = v^s_{0(j+1)} - v^s_{0j}
               \end{aligned}
               \right\}
              ~\text{ for all } j \Big\},
        \\
        &V^2_h = \big\{f = \sum_{i,j} f_{ij} T^2_{i,j} \in W^2_h : 
           ~ f_{0j} = 0 ~\text{ for all } j \big\},
        \end{aligned}
        \right.   
    \end{equation}
    where $\gamma_0$ is a real parameter that may differ 
    for each $\phi \in V^0_h$
    and correspond to its value at the pole, i.e.,
    $\phi(\bx_0) = \gamma_0$.
    Moreover, we have
    \begin{equation} \label{V0=C0}
    V^0_h = W^0_h \cap C^0(\Omega).
    \end{equation}
\end{thm}

\begin{oss}
   In light of \eqref{V0=C0}, we shall refer to \eqref{Vh} 
   as the $C^0$ complex.
\end{oss}
\medskip
\begin{thm}[$C^1$-$C^0$-$L^2$ polar spline complex] \label{thm:Useq}
   If $F$ is the analytical polar mapping \eqref{F_apol}, 
   then the conforming spaces 
   \eqref{Zh} for the $C^1$ complex \eqref{dR_U}, namely
   \begin{equation} \label{Uh}
      \left\{
      \begin{aligned}            
         U^0_h &\defeq W^0_h \cap H^1(\Omega) \cap C^1(\Omega)
          &&= V^0_h \cap C^1(\Omega),         
      \\
      U^1_h &\defeq W^1_h \cap H(\curl;\Omega) \cap C^0(\Omega) 
       &&= V^1_h \cap C^0(\Omega)
      \\
      U^2_h &\defeq W^2_h \cap L^2(\Omega)
       &&= V^2_h
      \end{aligned}
      \right.   
   \end{equation}   
   are characterized by the relations
   \begin{equation}\label{eq:UlCharacterization_apol}
      \left\{
      \begin{aligned}
      &U^0_h = \left\{
         \phi = \sum_{i,j} \phi_{ij} T^0_{i,j} \in W^0_h : 
         \left\{
         \begin{aligned}
         &\phi_{0j} = \gamma_0
         \\
         &\phi_{1j} = \gamma_0
         \end{aligned}
         \right\}
         \text{ for all $j$} 
      \right\},
      \\
      &U^1_h = \left\{
         \bv = \sum_{i,j} v^s_{ij} \pmb{T}^s_{i,j} 
         + \sum_{i,j} v^\theta_{ij} \pmb{T}^\theta_{i,j} \in W^1_h : 
         \left\{
         \begin{aligned}
         &v^s_{0j} = 0 
         \\
         &v^\theta_{0j} = 0
         \\
         &v^\theta_{1j} = v^s_{0(j+1)} - v^s_{0j}
         \end{aligned}
         \right\}
         \text{ for all $j$}
      \right\},
      \\
      &U^2_h = \big\{f = \sum_{i,j} f_{ij} T^2_{i,j} \in W^2_h : 
         ~ f_{0j} = 0 ~\text{ for all } j \big\},
      \end{aligned}
      \right.   
   \end{equation}   
   where $\gamma_0$ is a real parameter that may differ for each $\phi \in U^0_h$
   and corresponds to its value at the pole, i.e.,
   $\phi(\bx_0) = \gamma_0$.
   Moreover, for any $\phi \in U^0_h$ and any 
   $\bv \in U^1_h$, it holds 
   \begin{equation} \label{grad_phi_bv_0_apol}
      \bgrad \phi(\bx_0) = 0
      \quad \text{ and } \quad
      \bv(\bx_0) = 0.
   \end{equation}
   If $F$ is a spline mapping of the form \eqref{F_spol}, then
   the conforming spaces \eqref{Zh} for the complex
   \eqref{dR_U} are of the form 
   \begin{equation}\label{eq:UlCharacterization}
      \left\{
      \begin{aligned}
      &U^0_h = \left\{
         \phi = \sum_{i,j} \phi_{ij} T^0_{i,j} \in W^0_h : 
         \left\{
         \begin{aligned}
         &\phi_{0j} = \gamma_0
         \\
         &\phi_{1j} = \gamma_0 + \gamma_1 \cos \theta_j + \gamma_2 \sin \theta_j
         \end{aligned}
         \right\}
         \text{ for all $j$} 
      \right\},
      \\
      &U^1_h = \left\{
         \bv = \sum_{i,j} v^s_{ij} \pmb{T}^s_{i,j} 
         + \sum_{i,j} v^\theta_{ij} \pmb{T}^\theta_{i,j} \in W^1_h : 
         \left\{
         \begin{aligned}
         &v^s_{0j} = \eta_1 \cos \theta_j + \eta_2 \sin \theta_j
         \\
         &v^\theta_{0j} = 0
         \\
         &v^\theta_{1j} = v^s_{0(j+1)} - v^s_{0j}
         \end{aligned}
         \right\}
         \text{ for all $j$} 
      \right\},
      \\
      &U^2_h = \big\{f = \sum_{i,j} f_{ij} T^2_{i,j} \in W^2_h : 
         ~ f_{0j} = 0 ~\text{ for all } j \big\}.
      \end{aligned}
      \right.   
   \end{equation}
   Here, $\gamma_0$, $\gamma_1$, $\gamma_2$ are real parameters
   which may differ for each function $\phi \in U^0_h$, and
   correspond to its value and 
   gradient at the pole, namely
   \begin{equation} \label{phi_grad_phi_0}
      \phi(\bx_0) = \gamma_0,
      \qquad 
      \bgrad \phi(\bx_0) = \frac{1}{\rho_1} 
         \begin{pmatrix} \gamma_1 \\ \gamma_2\end{pmatrix}.
    \end{equation}
   Similarly, $\eta_1$ and $\eta_2$ are real parameters 
   which correspond to the value of each $\bv \in U^1_h$ at the pole, namely
    \begin{equation} \label{bv_0}
       \bv(\bx_0) = \frac{1}{\rho_1} 
       \begin{pmatrix} \eta_1 \\ \eta_2\end{pmatrix}
   \end{equation}
   where we remind that $\rho_1$ is the radius of the first control ring, 
   see~\eqref{F_spol}--\eqref{F_spol_P01}.

\end{thm}

Theorems~\ref{thm:Vseq} and \ref{thm:Useq} will be proven in Section~\ref{sec:charconf}.
To conclude this section we gather a few observations about these characterizations.

\medskip 
\begin{oss}
   While the spaces from the $C^0$ complex 
   $V^\ell_h$ in \eqref{eq:VlCharacterization} do not depend
   on the mapping itself,
   we observe in the $C^1$ complex \eqref{eq:UlCharacterization_apol}
   the spaces $U^0_h$ and $U^1_h$
   {\em do depend} on the
   precise form of the mapping $F$: indeed if the latter 
   is a spline mapping then the characterization of both spaces 
   involves its angular grid.
   Moreover, we see that a {\em locking effect} occurs 
   in the case of an analytical polar mapping: this effect, evidenced 
   by the relations \eqref{phi_grad_phi_0}, is due to the impossibility 
   of matching general variations of a spline fields with those of an 
   analytical mapping.
\end{oss}


\medskip 
\begin{oss}
   The extraction matrices $\bar{\bbol{E}}^k$ defined in 
   \cite[Sec.~3]{deepesh1} for $k = 0$ and $1$
   are consistent with the characterizations of 
   $V^0_h$ and $U^0_h$ in \eqref{eq:VlCharacterization} 
   and \eqref{eq:UlCharacterization},
   so that these spaces indeed coincide with the $C^0$ and $C^1$ 
   polar splines from \cite{deepesh1}, as one would expect.
   The spaces $U^1_h$ and $U^2_h$ also coincide with the polar spaces 
   in \cite{deepesh3}. Specifically, the first two basis functions of the 1-form polar space in \cite{deepesh3} satisfy the constraints 
   on the coefficients for $U^1_h$ as described in \eqref{eq:UlCharacterization},
   with $\eta_1 = -1/6$ and $\eta_2 = \pm \sqrt{3}/6$.
   
\end{oss}

\medskip 
\begin{oss}
    In \eqref{eq:UlCharacterization}, the characterization of the space
    $U^0_h$
    involves scalar parameters 
    $\gamma_0$, $\gamma_1$ and $\gamma_2$ 
    that can be computed from the tensor-product spline coefficients of 
    a given function $\phi = \sum_{i,j} \phi_{ij} T^0_{i,j}$ in $U^0_h$. 
    Indeed one has $\gamma_0 = \phi_{0j}$ for all $j$,
    and by using the relations \eqref{eq:trigo} from the Appendix, 
    one finds
    \begin{equation}
        \label{gamma_U0}
        \gamma_1 = \frac{2}{n_\theta}\sum_{j=0}^{n_\theta-1} 
        \phi_{1j}\cos \theta_j,
        \qquad
        \gamma_2 = \frac{2}{n_\theta}\sum_{j=0}^{n_\theta-1} 
            \phi_{1j}\sin \theta_j.
    \end{equation}
    Similarly, if
    $\bv = \sum_{i,j} v^s_{ij} \pmb{T}^s_{i,j} 
    + \sum_{i,j} v^\theta_{ij} \pmb{T}^\theta_{i,j}$ is a spline that 
    belongs to the $C^1$ space $U^1_h$, 
    then the parameters involved in the 
    characterization \eqref{eq:UlCharacterization} are given by the formulas
    \begin{equation}
        \label{eta_U1}
        \eta_1 = \frac{2}{n_\theta}\sum_{j=0}^{n_\theta-1} 
            v^s_{0j}\cos \theta_j,
        \qquad
        \eta_2 = \frac{2}{n_\theta}\sum_{j=0}^{n_\theta-1} 
            v^s_{0j}\sin \theta_j.
    \end{equation}
    We observe that both \eqref{gamma_U0} and \eqref{eta_U1} 
    may be seen as discrete Fourier transforms on the angular grid,
    and they provide some natural indications to define projection operators
    from the spaces $W^0_h$ and $W^1_h$ onto their respective subspaces 
    $U^0_h$ and $U^1_h$.
\end{oss}

\medskip
\begin{oss}
    Another consequence of \eqref{eq:UlCharacterization} 
    is that any $C^1$ tensor-product spline 
    $\phi \in U^0_h$ is of the form  
    $$
    \phi(\bx) 
        = \alpha_0(s) \gamma_0 
            + \alpha_1(s) \sum_{j=0}^{n_\theta-1} (\gamma_1 \cos \theta_j \per{B}_j(\theta) + \gamma_2 \sin \theta_j \per{B}_j(\theta))
            + r(s,\theta) 
    $$
   where 
   $ \alpha_0(s) = (B_0 + B_1)(s)$, 
   $\alpha_1(s) = B_1(s)$
   and
   $r(s,\theta) = \sum_{i \ge 2} \phi_{ij} B_i(s)\per{B}_j(\theta)$
   satisfy (using $\sum_i B_i(s) = 1$ and $B_i'(0) = 0$ for $i \ge 2$)
   $$
   \alpha_0(s) = 1 + \cO(s^2), 
   \qquad 
   \alpha_1(s) = sB_1'(0) + \cO(s^2), 
   \qquad 
   r(s, \theta) = \cO(s^2).
   $$
   This is reminiscent of the analysis in \cite{lewis_physical_1990}
   where it is shown that a smooth function decomposed in the form
   $$\psi(s,\theta) = \sum_{m \in \ZZ} a_m(s) e^{im\theta}$$
   must satisfy $a_m(s) = s^{\abs{m}} f_m(s^2)$ with $f_m$ a smooth function.
\end{oss}

\subsection{Geometric projections on the logical and polar domains}
\label{sec:geoproj}

On the logical domain, commuting projections can be obtained by interpolating geometric degrees of freedom such as point values, edge and cells integrals
\cite{Bochev_Hyman_2006_csd,Robidoux.2008.histo,Gerritsma.2011.spec}.
Specifically, one chooses interpolation nodes (a common choice for splines being the Greville points) 
along the logical axes, namely $0 = \zeta^s_{0} < \dots < \zeta^s_{n_s-1} = L$ and
$0 \le \zeta^\theta_{0} < \dots < \zeta^\theta_{n_\theta-1} < 2\pi$,
and next defines the logical projection operators $\hat \Pi^\ell_W$
by the relations
$$
\hat \Pi^\ell_W v \in \hat W^\ell_h \qquad \text{ and } \qquad 
\hat \arrsigma^\ell(\hat \Pi^\ell_W v) = \hat \arrsigma^\ell(v)
$$
where $\hat \arrsigma^\ell(v) \in \RR^{N^\ell_W}$, $N^\ell_W = \dim(\hat W^\ell_h)$, 
is the vector of geometric degrees of freedom 
corresponding to 
\begin{itemize}
   \item point values at nodes $\hat \ttn_{\bi} := (\zeta^s_{i_1}, \zeta^\theta_{i_2})$ 
   for $\ell=0$,
   \item line integrals over edges 
   $\hat \tte_{d,\bi} := [\hat \ttn_{\bi}, \hat \ttn_{\bi+\be_d}]$
   for $\ell=1$,
   \item integrals over cells 
   $\hat \ttc_{\bi} := [\hat \tte_{1,\bi}, \hat \tte_{1,\bi+\be_2}]$
   for $\ell=2$.
\end{itemize}
Here the square brackets $[\cdot]$ denote convex hulls,
the indices $\bi$ run over the whole interpolation grid 
and $\be_d$ is the canonical basis vector of $\RR^2$ along axis $d$.
We also refer to \cite[Sec.~6.2]{campos_pinto_variational_2022} or 
\cite[Sec.~B.1]{guclu_broken_2023} for more details, 
and a description of the Kronecker-product structure 
of the associated collocation matrix. 
Thanks to Stokes' formula these projections commute with the differential 
operators in logical variables, in the sense that 
\begin{equation} \label{cd_log}
	\begin{tikzpicture}[ampersand replacement=\&, baseline] 
	\matrix (m) [matrix of math nodes,row sep=3em,column sep=3em,minimum width=2em] {
      ~~ H^1(\hat \Omega) ~ \bbb
         \& ~~ H(\hcurl;\hat \Omega) ~ \bbb
            \& ~~ L^2(\hat \Omega) ~ \bbb
		\\
         ~~\hat W^0_h ~ \bbb
				\& ~~ \hat W^1_h ~ \bbb
					\& ~~ \hat W^2_h ~ \bbb
		\\
	};
	\path[-stealth]
	(m-1-1) edge node [above] {$\hbgrad$} (m-1-2)
			edge node [right] {$\hat \Pi^0_W$} (m-2-1)
	(m-1-2) edge node [above] {$\hcurl$} (m-1-3)
					edge node [right] {$\hat \Pi^1_W$} (m-2-2)
	(m-1-3) edge node [right] {$\hat \Pi^2_W$} (m-2-3)
	(m-2-1) edge node [above] {$\hbgrad$} (m-2-2)
	(m-2-2) edge node [above] {$\hcurl$} (m-2-3)
	;
	\end{tikzpicture}
\end{equation}
is a commuting diagram. 
Here the projection operators are indeed unbounded operators,
since the first two ones are only defined on functions with 
well-defined pointwise values and edge tangential traces, respectively.
One may consider for instance the domain spaces
\begin{equation} \label{D_hPiW}
D(\hat \Pi^0_W) = H^1(\hat \Omega) \cap C^0(\hat \Omega)
, \qquad
D(\hat \Pi^1_W) = H(\hcurl;\hat \Omega) \cap C^0(\hat \Omega)
, \qquad
D(\hat \Pi^2_W) = L^2(\hat \Omega)
\end{equation}
or the ones considered in \cite[Sec.~B.1]{guclu_broken_2023}.
The commuting relations read then
\begin{equation} \label{com_hPi_W}
   \begin{aligned}
      &\hat \Pi^1_W \hbgrad \hat \phi = \hbgrad \hat \Pi^0_W \hat \phi
         \quad &&\text{ for all } \hat \phi \in D(\hat \Pi^0_W) 
         \text{ such that } \hbgrad \hat \phi \in D(\hat \Pi^1_W),
      \\
      &\hat \Pi^2_W \hcurl \hat \bv = \hcurl \hat \Pi^1_W \hat \bv  
         \quad &&\text{ for all } \hat \bv \in D(\hat \Pi^1_W) 
         \text{ such that } \hcurl \hat \bv \in D(\hat \Pi^2_W).
   \end{aligned}
\end{equation}

On the polar domain we compose these 
projections with the pullbacks and pushforwards, i.e., define
$\Pi^\ell_W$ as in \eqref{Pi_W}
with domains $D(\Pi^\ell_W) := V^\ell \cap \cF^\ell D(\hat \Pi^\ell_W)$.
We then need to carefully study the commutation properties of the 
singular pullback and pushforward operators in order to 
specify those of the operators $\Pi^\ell_W$: this will be done in 
Section~\ref{sec:proof_cd_Pi_W}.

\subsection{Discrete conforming projections onto the $C^0$ complex} 
\label{sec:PV}

For the $C^0$ complex we define the projections 
by the following expressions (where $0 \le j < n_\theta$ 
is always arbitrary).
To project onto the space $V^0_h$ we define 
\begin{equation}\label{eq:PV0_def}
   P_V^0 : W^0_h \to W^0_h 
   \quad
   \text{ with } 
   \quad
   \left\{\begin{aligned}
      P_V^0 T_{0j}^0 &\coloneqq \frac{1}{n_\theta}\sum_{k=0}^{n_\theta-1} T_{0k}^0 
       \\ 
      P_V^0 T_{ij}^0 &\coloneqq T_{ij}^0 
         \quad \text{for $i \ge 1$}
      \end{aligned} \right.
\end{equation}
which, in terms of coefficients, amounts to setting
\begin{equation}\label{eq:PV0_coefs}
P^0_V: \phi \mapsto \bar \phi
\quad \text{ with } \quad 
\left\{\begin{aligned}
    \bar \phi_{0j} &\coloneqq \frac{1}{n_\theta}\sum_{k=0}^{n_\theta-1} \phi_{0k}
    \\
    \bar \phi_{ij} &\coloneqq \phi_{ij} \quad \text{ for } i \ge 1.
\end{aligned} \right.
\end{equation}
To project onto the space $V^1_h$ we define
\begin{equation}\label{eq:PV1_def}
   P_V^1 : W^1_h \to W^1_h 
   \quad
   \text{ with } 
   \quad
   \left\{\begin{aligned}
      P_V^1 \pmb{T}_{0j}^s 
         &\coloneqq 
         \pmb{T}_{0j}^s + \pmb{T}_{1(j-1)}^\theta - \pmb{T}_{1j}^\theta
      \\ 
      P_V^1 \pmb{T}_{ij}^s 
      &\coloneqq  \pmb{T}_{ij}^s\quad \text{for $i \ge 1$}
      \\[6pt]
      P_V^1 \pmb{T}_{0j}^\theta 
         &\coloneqq  0
      \\ 
      P_V^1 \pmb{T}_{1j}^\theta 
         &\coloneqq  0
      \\ 
      P_V^1 \pmb{T}_{ij}^\theta 
         &\coloneqq  \pmb{T}_{ij}^\theta \quad \text{for $i \ge 2$}
      \end{aligned} \right.
\end{equation}
which, in terms of coefficients, amounts to setting
\begin{equation}\label{eq:PV1_coefs}
   P^1_V: \bv \mapsto \bar \bv
   \quad \text{ with } \quad
   \left\{\begin{aligned}
       \bar v_{ij}^s 
         &\coloneqq v_{ij}^s \quad \text{for $i \ge 0$}
       \\[6pt]
       \bar v_{0j}^\theta 
         &\coloneqq 0
      \\ 
      \bar v_{1j}^\theta 
         &\coloneqq v_{0(j+1)}^s - v_{0j}^s 
      \\ 
      \bar v_{ij}^\theta 
         &\coloneqq v_{ij}^\theta \quad \text{ for $i \ge 2$}.
   \end{aligned} \right.
\end{equation}
To project onto the space $V^2_h$ we finally define   
\begin{equation}\label{eq:PV2_def}
   P_V^2 : W^2_h \to W^2_h 
   \quad
   \text{ with } 
   \quad
   \left\{\begin{aligned}
      P_V^2 T_{0j}^2 &\coloneqq T_{1j}^2
       \\ 
      P_V^2 T_{ij}^2 &\coloneqq T_{ij}^2 \quad \text{(for $i \ge 1$)}
   \end{aligned} \right.
\end{equation}
which, in terms of coefficients, amounts to setting
\begin{equation}\label{eq:PV2_coefs}
P^2_V: f \mapsto \bar f
\quad \text{ with } \quad 
\left\{\begin{aligned}
    \bar f_{0j} &\coloneqq 0
    \\
    \bar f_{1j} &\coloneqq f_{0j} + f_{1j}
    \\
    \bar f_{ij} &\coloneqq f_{ij} \quad \text{ for } i \ge 2.
\end{aligned} \right.
\end{equation}

\medskip
\begin{thm} \label{thm:PV}
   The above operators $P^\ell_V$ are projections onto the conforming 
   spaces $V^\ell_h$. 
   Moreover, the associated conforming-geometric projections 
   $\Pi^\ell_V \defeq P^\ell_V \Pi^\ell_W$,
   defined by composition with the projections \eqref{Pi_W},
   yield the commuting diagrams \eqref{cd} and \eqref{cd_3D} with $Z = V$.
\end{thm}

\begin{proof}[Proof (partial)]
The projection properties are easily verified by comparing their expressions in \eqref{eq:PV0_coefs}, \eqref{eq:PV1_coefs}, \eqref{eq:PV2_coefs}, with the characterization of the conforming spaces in \eqref{eq:VlCharacterization}
(remind that $P: A \to A$ is a projection onto $B \subset A$ 
if $P A \subset B$ and $P v = v$ holds for all $v \in B$).
The commutation properties will be proven in Section~\ref{sec:commutation}.
\end{proof}

By using the above formulas 
one may also write the (square) matrix representations of the projection 
operators in the tensor-product spline bases:
\begin{equation}\label{eq:PVlmatrices}\scalebox{.85}{$
\begin{array}{lll}
\bbol{P}_V^0 = \raisebox{-1.25cm}{\begin{tikzpicture}[scale=2]
\node at (.25,.75) {$\frac{1}{n_\theta}\bbol{J}$};
\node at (.25,-.15) {$0$};
\node at (1.15,-.15) {$\bbol{I}$};
\node at (1,.75) {$0$};
\draw (0,-.75) rectangle (1.75,1);
\draw (0,.5) -- (1.75,.5);
\draw (.5,-.75) -- (.5,1);
\end{tikzpicture}} &
\bbol{P}_V^1 = \raisebox{-3cm}{\begin{tikzpicture}[scale=2]
\node at (.75,.25) {$\bbol{I}$};
\draw (0,-.5) rectangle (1.5,1);
\draw[xstep=.5,ystep=.5] (0,-1.5) grid (.5,-.5);
\node at (.25,-.75) {$0$};
\node at (.25,-1.25) {$\mathring{\bbol{d}}$};
\draw (0,-2.5) rectangle (1.5,-.5);
\draw (.5,-2.5) -- (.5,-1.5);
\draw (.5,-1.5) -- (3.5,-1.5);
\node at (1,-1) {$0$};
\node at (1,-2) {$0$};
\node at (.25,-2) {$0$};
\draw (1.5,-2.5) rectangle (3.5,-.5);
\node at (2,-2) {$0$};
\node at (3,-2) {$\bbol{I}$};
\node at (3,-1) {$0$};
\draw (1.5,-.5) rectangle (3.5,1);
\draw (2.5,-2.5) -- (2.5,-.5);
\node at (2,-1) {$0$};
\node at (2.5,.25) {$0$};
\end{tikzpicture}} &
\bbol{P}_V^2 = \raisebox{-1.25cm}{\begin{tikzpicture}[scale=2]
\draw[xstep=.5,ystep=.5] (0,-.5) grid (.5,1);
\node at (.25,.25) {$\bbol{I}$};
\node at (.25,.75) {$0$};
\node at (.25,-.25) {$0$};
\node at (1,0) {$\bbol{I}$};
\node at (1,.75) {$0$};
\draw (0,-.5) rectangle (1.5,1);
\draw (.5,.5) -- (1.5,.5);
\end{tikzpicture}}
\end{array}$}
\end{equation}
The common identity blocks in $\bbol{P}_V^1$ and $\bbol{P}_V^2$ have size $n_\theta(n_s-2)$, the common identity blocks in $\bbol{P}_V^0$ and $\bbol{P}_V^1$ have size $n_\theta(n_s-1)$ and the smaller identity block in $P_V^2$ has size $n_\theta$. The blocks $\frac{1}{n_\theta}\bbol{J}$ and $\mathring{\bbol{d}}$ have both size $n_\theta$ as well. $\bbol{J}$ is the matrix whose entries are all equal to one while $\mathring{\bbol{d}}$ has the following structure and entries
\begin{equation}\label{eq:mathringd}
\mathring{\bbol{d}} = \begin{bmatrix}
-1&1\\
&\ddots&\ddots\\
&&-1&1\\
1&&&-1
\end{bmatrix}
\end{equation}
and it is the matrix associated to the derivative of periodic splines. 

\subsection{Discrete conforming projections onto the $C^1$ complex} \label{sec:PU}

For the $C^1$ complex we define the projections 
by the following expressions (where again 
$0 \le j < n_\theta$ is always arbitrary).
To project onto the space $U^0_h$ we define 
\begin{equation}\label{eq:PU0_def}
   P_U^0 : W^0_h \to W^0_h 
   \quad
   \text{ with } 
   \quad
   \left\{\begin{aligned}
      P_U^0 T_{0j}^0 
         &\coloneqq \frac{1}{n_\theta} \sum_{k=0}^{n_\theta-1} 
            (T_{0k}^0 + T_{1k}^0)
      \\ 
      P_U^0 T_{1j}^0
         &\coloneqq \frac{2}{n_\theta} \sum_{k=0}^{n_\theta-1} 
            \cos(\theta_k-\theta_j)T_{1k}^0
      \\ 
      P_U^0 T_{ij}^0 
         &\coloneqq T_{ij}^0 \qquad \text{for $i \ge 2$}
      \end{aligned} \right.
\end{equation}
which, in terms of coefficients, amounts to setting
\begin{equation}\label{eq:PU0_coefs}
P^0_U: \phi \mapsto \bar \phi
\quad \text{ with } \quad 
\left\{\begin{aligned}
    \bar \phi_{0j} &\coloneqq \frac{1}{n_\theta}\sum_{k=0}^{n_\theta-1} \phi_{0k}
    \\
    \bar \phi_{1j} &\coloneqq 
      \bar \phi_{0j} 
        + \frac{2}{n_\theta}\sum_{k=0}^{n_\theta-1} \cos(\theta_j-\theta_k)\phi_{1k}
      = \bar \phi_{0j} 
      + \bar \gamma_1 \cos\theta_j 
      + \bar \gamma_2 \sin\theta_j 
        \\
    \bar \phi_{ij} &\coloneqq \phi_{ij} \qquad \text{ for } i \ge 2
\end{aligned} \right.
\end{equation}
with 
\begin{equation} \label{gamma_P0U}
\bar \gamma_1 = \frac{2}{n_\theta}\sum_{k=0}^{n_\theta-1} \cos \theta_k \phi_{1k},
\qquad
\bar \gamma_2 = \frac{2}{n_\theta}\sum_{k=0}^{n_\theta-1} \sin \theta_k \phi_{1k}.
\end{equation}
To project onto the space $U^1_h$ we define
\begin{equation}\label{eq:PU1_def}
   P_U^1 : W^1_h \to W^1_h 
   \quad
   \text{ with } 
   \quad
   \left\{\begin{aligned}
      P_U^1 \pmb{T}_{0j}^s 
         &\coloneqq
         \pmb{T}_{1j}^s + 
         \frac{2}{n_\theta} \sum_{k=0}^{n_\theta-1} 
            \cos(\theta_k-\theta_j) \big(
            \pmb{T}_{0k}^s - \pmb{T}_{1k}^s
            + \pmb{T}_{1(k-1)}^\theta - \pmb{T}_{1k}^\theta
            \big)            
      \\ 
      P_U^1 \pmb{T}_{ij}^s 
         &\coloneqq 
         \pmb{T}_{ij}^s\quad \text{(for $i \ge 1$)}
      \\[6pt]
      P_U^1 \pmb{T}_{0j}^\theta &\coloneqq 0
      \\ 
      P_U^1 \pmb{T}_{1j}^\theta &\coloneqq 0
      \\ 
      P_U^1 \pmb{T}_{ij}^\theta 
         &\coloneqq \pmb{T}_{ij}^\theta 
         \quad \text{(for $i \ge 2$)}
      \end{aligned} \right.
\end{equation}
which, in terms of coefficients, amounts to setting
\begin{equation}\label{eq:PU1_coefs}
   P^1_U: \bv \mapsto \bar \bv
   \quad \text{ with } \quad
   \left\{\begin{aligned}
      \bar v_{0j}^s 
         &\coloneqq \frac{2}{n_\theta} \sum_{k=0}^{n_\theta-1} 
         \cos(\theta_j-\theta_k) v_{0k}^s
         = \bar \eta_1 \cos \theta_j + \bar \eta_2 \sin \theta_j
      \\
      \bar v_{1j}^s 
         &\coloneqq v_{1j}^s - \bar v_{0j}^s + v_{0j}^s
      \\
      \bar v_{ij}^s 
         &\coloneqq v_{ij}^s \qquad \text{for $i \ge 2$}
      \\[6pt]
       \bar v_{0j}^\theta 
         &\coloneqq 0
      \\ 
      \bar v_{1j}^\theta 
         &\coloneqq \bar v_{0(j+1)}^s - \bar v_{0j}^s 
      \\ 
      \bar v_{ij}^\theta 
         &\coloneqq v_{ij}^\theta \qquad \text{ for $i \ge 2$}
   \end{aligned} \right.
\end{equation}
with 
\begin{equation} \label{eta_P1U}
   \bar \eta_1 = \frac{2}{n_\theta} \sum_{k=0}^{n_\theta-1} \cos \theta_k v_{0k}^s
   \qquad
   \bar \eta_2 = \frac{2}{n_\theta} \sum_{k=0}^{n_\theta-1} \sin \theta_k v_{0k}^s.
\end{equation}
To project onto the space $U^2_h$ we finally set
$P_U^2 \defeq P_V^2$, see \eqref{eq:PV2_def}--\eqref{eq:PV2_coefs}.

\medskip 
\begin{oss}
   For both complexes $Z = V$ or $U$, we note that a compatibility relation holds
   between $P^0_Z$ and $P^1_Z$, namely
   \begin{equation} \label{P1_gradP0}
      P_Z^1 \pmb{T}_{0j}^s = \bgrad P_Z^0 T_{1j}^0  + \pmb{T}_{1j}^s.
   \end{equation}      
   This relation will be useful to prove the commuting properties.
\end{oss}

\medskip
\begin{thm} \label{thm:PU}
   The operators $P^\ell_U$ defined above are 
   projections onto the conforming spaces $U^\ell_h$. 
   Moreover, the associated conforming-geometric projections 
   $\Pi^\ell_U \defeq P^\ell_U \Pi^\ell_W$,
   defined by composition with the projections \eqref{Pi_W},
   yield the commuting diagrams \eqref{cd} and \eqref{cd_3D} with $Z = U$.
\end{thm}

\begin{proof}[Proof (partial)]
   One easily sees that $P^0_U$ maps into $U^0_h$ by considering its expression
   \eqref{eq:PU0_coefs} and the characterization in \eqref{eq:UlCharacterization}. 
   The projection property then follows by comparing \eqref{gamma_U0} and \eqref{gamma_P0U}. 
   Similarly, from \eqref{eq:PU1_coefs} one can deduce that $P^1_U$ maps into $U^1_h$, and the projection property is readily verified by comparing \eqref{eta_U1} 
   and \eqref{eta_P1U}. 
   The properties of $P^2_U$ are the same as those of $P^2_V$ since these two operators coincide.
   The commutation properties will be proven in Section~\ref{sec:commutation}.
\end{proof}

   The (square) matrices associated to the operators 
   $\{P_U^\ell\}_{\ell=0}^2 : W^\ell_h \to W^\ell_h$ 
   in the tensor-product bases are as follows:
   \begin{equation}\label{eq:PUlmatrices}\scalebox{.85}{$
   \begin{array}{lll}
   \bbol{P}_U^0 = \raisebox{-1.75cm}{\begin{tikzpicture}[scale=2]
   \draw[xstep=.5,ystep=.5] (0,0) grid (1,1);
   \node at (.25,.75) {$\frac{1}{n_\theta}\bbol{J}$};
   \node at (.25,.25) {$\frac{1}{n_\theta}\bbol{J}$};
   \node at (.75,.75) {$0$};
   \node at (.75,.25) {$\bbol{p}$};
   \node at (1.5,-.5) {$\bbol{I}$};
   \draw (1,-1) rectangle (2,0);
   \draw (0,-1) rectangle (2,1);
   \end{tikzpicture}} &
   \bbol{P}_U^1 = \raisebox{-3cm}{\begin{tikzpicture}[scale=2]
   \draw[xstep=.5,ystep=.5] (0,0) grid (.5,1);
   \node at (.25,.75) {$\bbol{p}$};
   \node at (.25,.25) {$\bbol{I}-\bbol{p}$};
   \node at (1,.75) {$0$};
   \node at (.25,-.25) {$0$};
   \node at (1,0) {$\bbol{I}$};
   \draw (0,-.5) rectangle (1.5,1);
   \draw (.5,.5) -- (1.5,.5);
   \draw (.5,-.5) -- (.5,0);
   \draw[xstep=.5,ystep=.5] (0,-1.5) grid (.5,-.5);
   \node at (.25,-.75) {$0$};
   \node at (.25,-1.25) {$\mathring{\bbol{d}}\bbol{p}$};
   \draw (0,-2.5) rectangle (1.5,-.5);
   \draw (.5,-2.5) -- (.5,-1.5);
   \draw (.5,-1.5) -- (3.5,-1.5);
   \node at (1,-1) {$0$};
   \node at (1,-2) {$0$};
   \node at (.25,-2) {$0$};
   \draw (1.5,-2.5) rectangle (3.5,-.5);
   \node at (2,-2) {$0$};
   \node at (3,-2) {$\bbol{I}$};
   \node at (3,-1) {$0$};
   \draw (1.5,-.5) rectangle (3.5,1);
   \draw (2.5,-2.5) -- (2.5,-.5);
   \node at (2,-1) {$0$};
   \node at (2.5,.25) {$0$};
   \end{tikzpicture}} &
   \bbol{P}_U^2 = \raisebox{-1.25cm}{\begin{tikzpicture}[scale=2]
   \draw[xstep=.5,ystep=.5] (0,-.5) grid (.5,1);
   \node at (.25,.25) {$\bbol{I}$};
   \node at (.25,.75) {$0$};
   \node at (.25,-.25) {$0$};
   \node at (1,0) {$\bbol{I}$};
   \node at (1,.75) {$0$};
   \draw (0,-.5) rectangle (1.5,1);
   \draw (.5,.5) -- (1.5,.5);
   \end{tikzpicture}}
   \end{array}$}
   \end{equation}
   where the large identity blocks $\bbol{I}$ have size $n_\theta(n_s-2)$ 
   while the small ones in $\bbol{P}_U^1$ and $\bbol{P}_U^2$ have size $n_\theta$. As above $\bbol{J}$ is the square matrix of size $n_\theta$ whose entries are all equal to one. 
   The matrices $\bbol{p}$ and $\bbol{q} \defeq \mathring{\bbol{d}}\bbol{p}$ are also square of size $n_\theta$, 
   with entries
   $$
   \bbol{p}_{j,k} = \frac{2}{n_\theta}\cos(\theta_j-\theta_k) 
   \qquad \text{ and } \qquad
   \bbol{q}_{j,k} = (\mathring{\bbol{d}}\bbol{p})_{j,k} = \frac{2}{n_\theta}(\bbol{p}_{j+1,k} - \bbol{p}_{j,k})
   $$
   for $j,k = 0, \ldots, n_\theta-1$.
   We note that $\bbol{p}$ is symmetric, and that $\bbol{p}$ and $\bbol{q}$ are  both Toeplitz matrices
   since the angles are regular according to \eqref{F_spol_P01}. 

\medskip 
\begin{oss}
   The projection operator $\Pi^2_V$ 
   (and $\Pi^2_U$, as they coincide)
   is mass preserving, in the sense that 
   $$
   \int_\Omega \Pi^2_V f = \int_\Omega f.
   $$ 
   This follows from the fact that the logical 
   projection $\hat \Pi^2_W$ is mass preserving,
   as well as the pushforward and pullback operators
   $\cF^2$, $\cB^2$, and the conforming projection $P^2_V$
   characterized by \eqref{eq:PV2_coefs}:
   the latter property holds since
   all M-splines (and in particular the first two) 
   have the same integral, see \eqref{int_M} from the Appendix.
\end{oss}

In the remainder of this section we
specify the matrix representation of the discrete 
differential operators involved in the diagram \eqref{cd}, 
as well as regularized mass matrices 
for the non-conforming spaces $W^\ell_h$.

\subsection{Discrete derivative matrices}

On the conforming spaces $Z^0_h$ and $Z^1_h$ (again with $Z = V$ or $U$), 
the differential operators commute with their logical counterparts. 
As a consequence, we may use the matrices of the latter 
in the tensor-product B-spline bases \eqref{hT_bases}, defined as
$$
\mat{G} = \hat{\mat{G}} = \Big( \hat\sigma^1_\mu( \hbgrad \hat \Lambda^0_\nu )\Big)_{(\mu, \nu) \in \cI^1_h \times \cI^0_h}
\qquad 
\mat{C} = \hat{\mat{C}} = \Big( \hat\sigma^2_\mu( \hbgrad \hat \Lambda^1_\nu )\Big)_{(\mu, \nu) \in \cI^2_h \times \cI^1_h}
$$
where we have used the generic notation $\hat \Lambda^\ell_\mu$ for the tensor-product B-splines of $\hat W^\ell_h$, 
and $\hat\sigma^\ell_\mu$ for the 
corresponding degrees of freedom, with index sets
$$
\begin{aligned}
   \cI^0_h &= \{ (i,j) : 0 \le i < n_s, 0 \le j < n_\theta\}
   \\
   \cI^1_h &= \{ (s,i,j) : 0 \le i < n_s-1, 0 \le j < n_\theta\} \, \cup \, \{ (\theta,i,j) : 0 \le i < n_s, 0 \le j < n_\theta\}
   \\
   \cI^2_h &= \{ (i,j) : 0 \le i < n_s-1, 0 \le j < n_\theta\}.
\end{aligned}
$$
These matrices have very simple expressions,
indeed the gradient and curl of basis functions in logical variables
read
\begin{equation}
   \label{hgrad_basis}
   \hbgrad (B_i \otimes \per{B}_j) 
   = \begin{pmatrix}
      (M_{i-1} - M_i) \otimes \per{B}_j 
      \\
      B_i \otimes (\per{M}_{j-1} - \per{M}_j) 
   \end{pmatrix}
\end{equation}
and (for arbitrary coefficients $a$ and $b$)
\begin{equation}
   \label{hcurl_basis}
   \hcurl \begin{pmatrix}
   a M_i \otimes \per{B}_j 
   \\
   b B_i \otimes \per{M}_j 
   \end{pmatrix}
  = 
  b (M_{i-1} - M_i) \otimes \per{M}_j - a M_i \otimes (\per{M}_{j-1} - \per{M}_j).
\end{equation}
With these respective formulas we find
$$
\mat{G}_{(s,l,k),(i,j)} = (\delta_{l,i-1} - \delta_{l,i})\delta_{k,j},
\qquad 
\mat{G}_{(\theta,l,k),(i,j)} = \delta_{l,i}(\delta_{k,j-1} - \delta_{k,j})
$$
and 
$$
\mat{C}_{(l,k),(s,i,j)} = - \delta_{l,i} (\delta_{k,j-1} - \delta_{k,j}),
\qquad 
\mat{C}_{(l,k),(\theta,i,j)} = (\delta_{l,i-1} - \delta_{l,i})\delta_{k,j}.
$$
On the full (non-conforming) tensor-product spaces, our 
differential operators are then simply represented 
by the matrix products 
$$
\mat{G} \mat{P}^0 ~\text{ for } ~\bgrad P^0_Z: ~ W^0_h \to W^1_h 
\qquad \text{ and } \qquad
\mat{C} \mat{P}^1 ~\text{ for } ~\curl P^1_Z: ~ W^1_h \to W^2_h.
$$

\subsection{Regularised mass matrices}
\label{sec:regmass}

Because the spaces $W^1_h$ and $W^2_h$ are not subspaces of $L^2(\Omega)$, 
the broken-FEEC discretization described in Section~\ref{sec:conga} must be 
amended, as pointed out in Remark~\ref{rem:mass_unbounded}. 
In particular, the mass matrices $\mat{M}^\ell$ 
associated with the B-spline basis \eqref{T_bases} for $\ell = 1, 2$, 
are ill-defined as they contain unbounded entries.

Our solution to this issue is to equip the corresponding spaces $W^\ell_h$
with discrete inner products 
\begin{equation}
   \label{ssp}
   \ssprod{u}{v}_{h,\ell} \defeq 
   \sprod{P^\ell_Z u}{P^\ell_Z v}_{L^2(\Omega_h)}
   + \sprod{(I-\hat P^\ell_Z) \hat u}{(I-\hat P^\ell_Z) \hat v}_{\hat W^\ell_h},
   \qquad u, v \in W^\ell_h.
\end{equation}
Here 
$\hat u$ and $\hat v \in \hat W^\ell_h$ are the pullbacks of $u$ and $v$, 
$\hat P^\ell_Z := \cB^\ell P^\ell_Z$ is the conforming projection applied on 
the logical spline space $\hat W^\ell_h$ 
(which is equivalent in terms of spline coefficients)
and $\sprod{\cdot}{\cdot}_{\hat W^\ell_h}$ is an arbitrary scalar product
on $\hat W^\ell_h$, for instance 
$
\sprod{\hat u}{\hat v}_{\hat W^\ell_h} 
   = \sum_{\mu} u_\mu v_\mu
$
that involves the B-spline coefficients of $\hat u$ and $\hat v$.
We note that in \eqref{ssp} the $L^2$ product is well-defined since it 
only involves conforming functions. It is then easily verified that
the product $\ssprod{\cdot}{\cdot}_{h,\ell}$ is symmetric positive definite,
and that it coincides with the $L^2$ product on the conforming space $Z^\ell_h$, 
due to the projection properties of $P^\ell_Z$.
The resulting ``regularized'' mass matrix reads
\begin{equation}
   \label{tM}
   \tmat{M}^\ell \defeq 
   (\mat{P}^\ell)^T\mat{M}^\ell \mat{P}^\ell
   + (\mat{I}^\ell-\mat{P}^\ell)^T (\mat{I}^\ell-\mat{P}^\ell)
\end{equation}
where $\mat{P}^\ell$ and $\mat{I}^\ell$ are the matrices of the projection $P^\ell_Z$ and identity operator in $W^\ell_h$, respectively.

In Sections~\ref{sec:charconf}, \ref{sec:pbpf} and \ref{sec:commutation}
we proceed to prove our main results, namely 
Theorems~\ref{thm:Vseq} and \ref{thm:Useq} which characterize the conforming 
polar spline spaces, and Theorems~\ref{thm:PV} and \ref{thm:PU} which state 
the commutation properties of our conforming projection operators. 
Finally, in Section~\ref{sec:num} we will perform some numerical experiments to
validate the practical stability and accuracy properties of our broken-FEEC
operators.

\section{Proofs of Theorems~\ref{thm:Vseq} and \ref{thm:Useq}}
\label{sec:charconf}

In this section we demonstrate the explicit characterizations 
of the conforming polar spaces in terms of their tensor-product spline coefficients,
namely \eqref{eq:VlCharacterization} and 
\eqref{eq:UlCharacterization_apol}--\eqref{eq:UlCharacterization}.
For every value of the space index $\ell \in \{0, 1, 2\}$ we first study the conforming space $V^\ell_h = W^\ell_h \cap V^\ell$, 
and then that of its subspace $U^\ell_h = W^\ell_h \cap U^\ell \subset V^\ell_h$.

\subsection{Study of the conforming spaces in $V^0$ and $U^0$}

We begin by observing that the equality \eqref{V0=C0}, namely $V^0_h = W^0_h \cap C^0(\Omega)$,
is straightforward using the interpolatory property of B-splines at $s=0$, 
namely \eqref{BM_0}--\eqref{BM_0p} from the Appendix, and the fact that the 
$\cF^0$ pushforward is a simple change of variable.

Next, let $\hat \phi = \sum_{ij} \phi_{ij} \hT_{ij}^0 \in \hat W^0_h$
and write $\phi = \cF^0 \hat \phi$.
Using the commutation relation \eqref{com_pf} 
we find for all $\bx = F(s,\theta) \in \Omega_0$
(i.e., for $s > 0$)
$$
\bgrad \phi(\bx) = \Big(\cF^1 \sum_{ij} \phi_{ij} \hbgrad \hT_{ij}^0\Big)(\bx) 
    = J_F^{-T} (s, \theta)\begin{pmatrix}
   \sum_i B_i'(s) \alpha_i(\theta) \\
   \sum_i B_i(s) \alpha_i'(\theta)
\end{pmatrix}
$$
with 
$\alpha_i(\theta) = \sum_j \phi_{ij} \per{B}_j(\theta)$.
Again from Equations \eqref{BM_0}--\eqref{BM_0p}, we see that B-splines behave at the pole like 
(writing $\mu = M_0(0) = B_1'(0) = -B_0'(0)$)
\begin{equation} \label{BM01i}
\begin{cases}
   B_0(s) = 1 + \cO(s) &
   \\
   B'_0(s) 
           = -\mu + \cO(s) &
   \\
   M_0(s) = \mu + \cO(s) &
\end{cases}
\quad
\begin{cases}
   B_1(s) 
         = s(\mu + \cO(s)) &
   \\ 
   B'_1(s) 
      = \mu + \cO(s) &
   \\
   M_1(s) = s(M'_1(0) + \cO(s)) &
\end{cases}
~
\text{and }
\quad
\begin{cases}
   B_i(s) = s\cO(s) &
   \\ 
   B'_i(s) = \cO(s) &
   \\
   M_i(s) = s\cO(s) & 
\end{cases}
\text{ for } i \ge 2,
\end{equation}
so that we have,
$$
\bgrad \phi(\bx) = 
   J_F^{-T}(s, \theta) \begin{pmatrix}
      -\mu(\alpha_0(\theta) - \alpha_1(\theta)) + \cO(s)
      \\
      (1+\cO(s))\alpha'_0(\theta) + s\mu\alpha'_1(\theta) + s\cO(s)
   \end{pmatrix}
$$
with
\begin{equation} \label{JF-T}
   J_F^{-T}(s, \theta) = \frac{1}{\det J_F}\begin{pmatrix}
      \partial_2 F_2 & -\partial_1 F_2 \\
      -\partial_2 F_1 & \partial_1 F_1
  \end{pmatrix}(s, \theta)
  = \frac{1}{s D(\theta)}
    \begin{pmatrix}
      s(S'(\theta) + \cO(s)) & -S(\theta) + \cO(s) \\
      -s(C'(\theta) + \cO(s)) & C(\theta) + \cO(s)
  \end{pmatrix}
\end{equation}
where the second equality follows from Proposition~\ref{prop:singFpol}.
This gives
\begin{equation}  
\bgrad \phi(\bx) 
   =
   \frac{1}{s D(\theta)}\begin{pmatrix}
      -(S(\theta) + \cO(s))\alpha_0'(\theta)
      - s\Big(\mu (S(\theta) \alpha_1' + S'(\theta)(\alpha_0 - \alpha_1)) + \cO(s)\Big)
      \\[4pt]
      (C(\theta) + \cO(s))\alpha_0'(\theta)
      + s\Big(\mu (C(\theta) \alpha_1' + C'(\theta)(\alpha_0 - \alpha_1)) + \cO(s)\Big)
\end{pmatrix}.
\end{equation}
By writing the $L^2$ norm on the polar domain as
\begin{equation}
   \label{grad_phi_L2}
   \norm{\bgrad \phi}^2_{L^2(\Omega)} = \int_0^L \int_0^{2\pi} \abs{\bgrad \phi(F(s,\theta))}^2 \det J_F(s,\theta) \rmd s \rmd \theta,
\end{equation}
we then see that $\phi \in H^1(\Omega)$ if and only if $\alpha_0'(\theta) = 0$, 
which amounts to $\phi_{0j} = \phi_{00}$ for all $j$
(as in \eqref{eq:VlCharacterization} we will denote this constant coefficient by $\gamma_0$) 
by using the elementary properties of the B-splines.
This proves the characterization of $V^0_h$ in 
\eqref{eq:VlCharacterization}.
Turning to $U^0_h$, we now observe that for $\phi \in V^0_h$
we have
\begin{equation}
   \label{grad_phi}
   \bgrad \phi(\bx) 
   =
   \frac{1}{D(\theta)}\begin{pmatrix}
      -(S(\theta) \beta'(\theta) - S'(\theta)\beta(\theta))
      \\[4pt]
       C(\theta) \beta'(\theta) - C'(\theta)\beta(\theta)
\end{pmatrix}
+ \cO(s)
\end{equation}
where we have written 
\begin{equation}\label{beta}
\beta(\theta) 
   = \mu(\alpha_1(\theta) - \alpha_0)
   = \mu \sum_j(\phi_{1j} - \phi_{0j}) \per{B}_j(\theta).
\end{equation}
This shows that $\bgrad \phi$ is bounded on $\Omega$. Its continuity then amounts to 
the existence of a single-valued gradient at the pole 
$\bx_0 = F(0,\theta)$, say 
$\bG = \bgrad \phi(F(0,\theta)) \in \RR^2$. 
Thus, $\phi$ is in $C^1(\Omega)$ if and only if for all $\theta$, we have
\begin{equation} \label{beta_conds}   
\begin{cases}
-S(\theta) \beta'(\theta) + S'(\theta)\beta(\theta)
= G_1 D(\theta)
\\[4pt]
\phantom{-} C(\theta) \beta'(\theta) - C'(\theta)\beta(\theta)
 = G_2 D(\theta)
\end{cases}
\quad \text{ i.e., } \quad
\begin{cases}
   S'(\beta - G_1 C)
   = S (\beta - G_1 C)' 
   \\[4pt]
   C'(\beta - G_2 S )
    = C(\beta - G_2 S)' 
\end{cases}   
\end{equation}
where the relation $D = CS' - SC'$ follows from Proposition~\ref{prop:singFpol}.
It is easy to see that equations \eqref{beta_conds} have a unique solution, namely
\begin{equation} \label{C1_alpha}
   \beta(\theta) = G_1 C(\theta) + G_2 S(\theta).
\end{equation}
With the analytical polar mapping \eqref{F_apol}, this yields
$\beta(\theta) = G_1 \cos \theta + G_2 \sin \theta$
which, since $\beta$ is piecewise polynomial,
is only possible for 
$\beta(\theta) = 0 = G_1 = G_2$:
This proves the first relation in \eqref{grad_phi_bv_0_apol}, and 
from \eqref{beta} it also shows that
$\phi_{1j} = \phi_{0j}$ for all $j$, 
which, using $\phi_{0j} = \gamma_0$,
proves the characterization of $U^0_h$ in 
\eqref{eq:UlCharacterization_apol}.
With the spline mapping \eqref{F_spol}, the expressions \eqref{CS_spol} yield
$$
\beta(\theta) 
=  G_1 \rho_1 \mu 
   \sum_{j=0}^{n_\theta-1} \cos\theta_j \per{B}_j(\theta)
   +
   G_2 \rho_1 \mu \sum_{j=0}^{n_\theta-1} \sin\theta_j \per{B}_j(\theta)
$$
which, according to \eqref{beta}, amount to
$
\phi_{1j} - \phi_{0j} = 
   G_1 \rho_1 \cos\theta_j
   +
   G_2 \rho_1 \sin\theta_j.
$
Using again that $\phi_{0j} = \gamma_0$ for all $j$, 
this proves the characterization of $U^0_h$ in
\eqref{eq:UlCharacterization} with 
$\gamma_d = G_d \rho_1$, which in turn also proves 
relation \eqref{phi_grad_phi_0}.

\subsection{Study of the conforming spaces in $V^1$ and $U^1$}

Consider next 
$\hat \bv \in \hat W^1$,
written in the form
$$
\hat \bv(s,\theta)
   =  \sum_{i,j} v^s_{i} \hbT^s_{ij}(s,\theta) 
         + \sum_{i,j} v^\theta_{ij} \hbT^\theta_{ij}(s,\theta)
   =
   \sum_{i} \begin{pmatrix} 
      M_i(s) \alpha^s_i(\theta) \\
      B_i(s) \alpha^\theta_i(\theta)
   \end{pmatrix}
   \quad \text{ with } \quad 
   \begin{cases}
      \alpha^s_i(\theta) = \sum_j v^s_{ij} \per{B}^s_{j}(\theta) 
      \\
      \alpha^\theta_i(\theta) = \sum_{j} v^\theta_{ij} \per{M}_{j}(\theta)
   \end{cases}
$$
for which the logical curl is
$
\hcurl \hat \bv(s,\theta) 
   = 
   \sum_{i} 
   \big(B'_i(s) \alpha^\theta_i(\theta)
   - M_i(s) (\alpha^s_i)'(\theta)\big).
   $
On $\Omega_0$ (i.e., for $\bx = F(s,\theta)$ with $s > 0$), 
the pushforward $\bv \defeq \cF^1 \hat \bv$ writes
\begin{equation}
   \bv(\bx) = J_F^{-T} \hat \bv (s, \theta)
   = \frac{1}{s D(\theta)} \sum_{i} 
   \begin{pmatrix}
     s(S'(\theta) +\cO(s)) M_i(s) \alpha^s_i(\theta)
         + (-S(\theta) + \cO(s)) B_i(s) \alpha^\theta_i(\theta)
         \\
     -s(C'(\theta) +\cO(s)) M_i(s) \alpha^s_i(\theta)
      +  (C(\theta) + \cO(s)) B_i(s) \alpha^\theta_i(\theta)
   \end{pmatrix}
\end{equation}
and by using \eqref{com_pb_curl} and Proposition~\ref{prop:singFpol}, its curl is
\begin{equation}
   \label{curl_bv}
   \curl \bv(\bx) 
   = (\cF^2 \hcurl \hat \bv)(\bx) 
   = 
   (\det J_F )^{-1} \hcurl \hat \bv (s,\theta)
   = \frac{1+\cO(s)}{s D(\theta)} \sum_{i} 
   \big(B'_i(s) \alpha^\theta_i(\theta)
   - M_i(s) (\alpha^s_i)'(\theta)\big).
\end{equation}
Using the behaviour of B-splines at the pole recalled in \eqref{BM01i}, 
we find that (with implicit dependence on $\theta$)
\begin{equation} \label{bv}
\bv(\bx) = \frac{1}{s D} 
\begin{pmatrix}
   -(S + \cO(s)) \alpha^\theta_0
   - s \mu(S  \alpha^\theta_1 - S' \alpha^s_0)
   \\
    (C + \cO(s)) \alpha^\theta_0
       + s \mu (C \alpha^\theta_1 - C' \alpha^s_0)
\end{pmatrix}
+ \cO(s)
\end{equation}
and 
$$
\curl \bv(\bx) =  \frac{\mu \big(\alpha^\theta_1 - \alpha^\theta_0 - (\alpha^s_0)'\big)}{s D}  +\cO(s) 
$$
so that, reasoning as above (see \eqref{grad_phi_L2}), gives
$$
\left.
\begin{aligned}
   \bv \in L^2(\Omega)  
      & \iff ~ \alpha^\theta_0 = 0 
   \\
   \curl \bv \in L^2(\Omega) 
      & \iff ~ \alpha^\theta_1 - \alpha^\theta_0 - (\alpha^s_0)' = 0
\end{aligned}
\right\}
~ \iff ~ 
\left\{
\begin{aligned}
   &v_{0j}^\theta = 0 ~ 
   \\
   &v_{1j}^\theta = v_{0(j+1)}^s - v_{0j}^s
\end{aligned}
\right.
~~  \text{ (for all $j$), }
$$
which is precisely the characterization of $V^1_h$ in 
\eqref{eq:VlCharacterization}.
Next for $\bv \in V^1_h$, we see that \eqref{bv} yields 
$$
\bv(\bx) = \frac{1}{D(\theta)} 
\begin{pmatrix}
   - \mu(S (\alpha^s_0)' - S'\alpha^s_0)(\theta)
   \\
   \mu(C (\alpha^s_0)' - C' \alpha^s_0)(\theta)
\end{pmatrix}
+ \cO(s)
$$
so that the behaviour at the pole is the same as 
$\bgrad \phi$ in \eqref{grad_phi}, here with $\beta = \mu \alpha^s_0$.
In particular, the same arguments as above show that 
$$
\bv \in C^0(\Omega)
~ \iff ~ 
\text{for all $j$}, \quad
\begin{cases}
   \alpha^s_{0j} = 0  & \text{ if $F$ is the analytical mapping \eqref{F_apol}}
   \\
   \alpha^s_{0j} = \eta_1 \cos\theta_j + \eta_2 \sin\theta_j
   & \text{ if $F$ is the spline mapping \eqref{F_spol}}
\end{cases}
$$
with constants corresponding to $\eta_d = v_d(\bx_0) \rho_1$.
This proves the characterization of $U^1_h$ in 
\eqref{eq:UlCharacterization_apol} and 
\eqref{eq:UlCharacterization} for the analytical and spline mappings, respectively.
It also shows the second relation in \eqref{grad_phi_bv_0_apol} 
for the analytical mapping, and the relation \eqref{bv_0} 
for a spline mapping.

\subsection{Study of the conforming spaces in $V^2 = U^2$}

Finally, let $\hat f \in \hat W^2_h$ in the form
$\hat f(s,\theta) = \sum_{i} M_i(s) \tilde \alpha_i(\theta)$
with $\tilde \alpha_i(\theta) \defeq \sum_j f_{ij} \per{M}_j(\theta)$,
and write
\begin{equation*}
   \begin{aligned}
   f(\bx) 
   &= (\cF^2 \hat f)(\bx) 
   = (\det J_F)^{-1} \hat f (s,\theta)
   = \frac{1+\cO(s)}{s D(\theta)} 
   \sum_{i} M_i(s) \tilde \alpha_i(\theta)
   \\
   &= \frac{1}{s D(\theta)} 
   (M_0(0) \tilde \alpha_0(\theta) + \cO(s)).   
   \end{aligned}
\end{equation*}
By reasoning as above this readily shows that
$$
f \in L^2(\Omega) ~ \iff ~ \tilde \alpha_0 = 0 ~ \iff ~ f_{0j} = 0 ~\text{ for all $j$,}
$$ 
which proves the characterization of $V^2_h$ and $U^2_h$ in \eqref{eq:VlCharacterization}
and \eqref{eq:UlCharacterization}.

\section{Pullbacks and pushforwards with polar singularity}
\label{sec:pbpf}

In this section we establish a few properties of the pullback and pushforward
operators associated with polar mappings in the sense of 
Definition~\ref{def:singpol}.
In Section~\ref{sec:commutation}, these properties will be instrumental 
in proving that our polar conforming-geometric 
projections indeed commute with the 
differential operators as shown in Diagram~\ref{cd}.

\subsection{Properties of the pullback operators}
\label{sec:pb}

We first verify that the pullbacks \eqref{pb} commute 
with the differential operators.

\medskip
\begin{prop} \label{prop:com_pb}
   Let $\phi$ in $H^1(\Omega)$ and $\bv \in H(\curl;\Omega)$. 
   Then the following relations hold in $L^2(\hat \Omega)$:
   \begin{align}
      \cB^1 \bgrad \phi &= \hbgrad \cB^0 \phi \label{com_pb_grad}
      \\
      \cB^2 \curl \bv &= \hcurl \cB^1 \bv \label{com_pb_curl}.
   \end{align}
\end{prop}
\begin{proof}
   Considering first the case of a smooth $\phi \in C^1(\overline \Omega)$,
   we see that its pullback $\hat \phi$
   is $C^1$ on the logical domain, and we have
   $
   \hat \partial_i \hat \phi 
      = \sum_j  ((\partial_j \phi) \circ F ) \hat \partial_i F_j
      = (J_F^T (\bgrad \phi \circ F))_i 
   $
   in a pointwise sense: this shows 
   \eqref{com_pb_grad} for smooth $\phi$. 
   For $\phi \in H^1(\Omega)$ the equality follows from the density of
   $C^1(\overline \Omega)$, together with the continuity of the pullbacks
   betwen $L^2$ spaces, see Proposition~\ref{prop:pb_Lq}.

   Turning to \eqref{com_pb_curl} we consider again a smooth 
   $\bv \in C^1(\overline \Omega)$. Because $F$ is only assumed $C^1$, its 
   Jacobian matrix $J_F$ is only continuous
   and the pullback $\hat \bv := \cB^1\bv = J_F^T (\bv \circ F)$
   may not be $C^1$, but since $(\bv \circ F)$ is $C^1$ one can compute
   in distribution's sense
   $$
   \hat \partial_i \hat v_j 
      = \hat \partial_i \sum_k (\hat \partial_j F_k) (v_k \circ F )
      = \sum_k \Big((\hat \partial^2_{ij} F_k) (v_k \circ F ) 
      + \sum_l (\hat \partial_i F_l \hat \partial_j F_k)  ((\partial_l v_k) \circ F )
      \Big),
   $$
   indeed each term in the second sum is a well-defined distribution.
   From this, we infer that
   $$
   \hcurl \hat \bv = \hat \partial_1 \hat v_2 - \hat \partial_2 \hat v_1 
   = \sum_{k,l} (\hat \partial_1 F_l \hat \partial_2 F_k - \hat \partial_2 F_l \hat \partial_1 F_k) 
      ((\partial_l v_k) \circ F )      
   = \det J_F ((\curl \bv) \circ F ) = \cB^2 \curl \bv.
   $$
   Since the right hand side is a continuous function from Proposition~\ref{prop:pb_Lq}, 
   this shows that $\hcurl \hat \bv =  \cB^2 \curl \bv$ holds in a pointwise sense.
   The result for $\bv \in H(\curl;\Omega)$ follows again from the density of 
   smooth functions and the continuity of the pullbacks in $L^2$, see 
   Proposition~\ref{prop:pb_Lq}.
\end{proof}

Similar relations hold for the complex 
\begin{equation} \label{dR_Om_dual}
  H(\bcurl;\Omega) \xrightarrow{ \mbox{$~ \bcurl ~$}}
    H(\div;\Omega) \xrightarrow{ \mbox{$~ \div ~$}}
      L^2(\Omega)
\end{equation}
with pushforward operators $\tilde \cF^\ell: C^0(\hat \Omega_0) \to C^0(\overline \Omega_0)$ defined as
\begin{equation} \label{pf_tilde}
  \left\{
  \begin{aligned}
  &\tilde \cF^0 : \hat g \mapsto g := \hat g \circ F^{-1}
  \\
  &\tilde \cF^1 : \hat \bw \mapsto \bw :=  \big((\det J_F^{-1}) J_F \hat \bw \big)\circ F^{-1}
  \\
  &\tilde \cF^2 : \hat \psi \mapsto \psi :=  \big(\det J_F^{-1} \hat \psi \big)\circ F^{-1}
  \end{aligned}
  \right.
\end{equation}
and inverse (pullback) operators $\tilde \cB^\ell: C^0(\overline \Omega_0) \to C^0(\hat \Omega_0)$
with expressions
\begin{equation} \label{pb_tilde}
   \left\{
   \begin{aligned}
   &\tilde \cB^0 : \phi \mapsto \hat \phi := \phi \circ F
   \\
   &\tilde \cB^1 : \bw \mapsto \hat \bw :=   (\det J_F )J_F^{-1} (\bw \circ F)
   \\
   &\tilde \cB^2 : f \mapsto \hat f :=  \det J_F (f \circ F).
   \end{aligned}
   \right.
\end{equation}
 
By reasoning as above, we indeed obtain the following result.
\medskip
\begin{prop} \label{prop:com_pb_tilde}
   The pullbacks $\tilde \cB^\ell$ map continuously 
   $L^q(\Omega)$ to $L^q(\hat \Omega)$, $1 \le q \le \infty$,
   and $C^0(\overline \Omega)$ to $C^0(\hat \Omega)$.
   Moreover, for any $g \in H(\bcurl;\Omega)$ and $\bv \in H(\div;\Omega)$,
   the following relations hold in $L^2(\hat \Omega)$:
   \begin{align}
      \tilde \cB^1\bcurl g &= \hbcurl \tilde \cB^0 g 
      \label{com_pb_bcurl}
      \\
      \tilde \cB^2 \div \bw &= \hdiv \tilde \cB^1 \bw.
      \label{com_pb_div}
   \end{align}
\end{prop}
\smallskip



The following proposition gathers additional properties of the pullback
operators, that will be useful in the next section.

\medskip
\begin{prop} \label{prop:pb_0}
   One has 
   \begin{equation} \label{cov}
      \int_{\Omega} \phi f
         = \int_{\hat \Omega} \Big(\cB^0 \phi \Big)\, \Big(\tilde \cB^2 f\Big)
         \qquad \text{ and } \qquad
         \int_{\Omega} \bv \cdot \bw
         = \int_{\hat \Omega} \Big(\cB^1 \bv \Big)\cdot \Big( \tilde \cB^1 \bw \Big)
   \end{equation}
   for any functions $\phi, \bv \in L^1(\Omega)$ and $f, \bw$ in $L^\infty(\Omega)$. 
   Moreover if $\phi$ and $\bv$ are continuous on $\Omega$,
   then we have
   \begin{equation}
      \label{bc_pb}
      \big(\cB^0 \phi\big) (0,\theta) = \phi(\bx_0),
      \quad
      \big(\hat \bn \times \cB^1 \bv \big)(0,\theta) = 0
      \quad \text{ and } \quad
      \big(\hat \bn \cdot \tilde \cB^1 \bw\big)(0,\theta) = 0 
   \end{equation}
   for all $\theta$, where $\hat \bn = (-1, 0)^T$ is the outward normal vector on the logical boundary $\{s=0\}$.
\end{prop}
\smallskip

\begin{proof}
   Both equalities in \eqref{cov} are obtained with a 
   change of variable.
   For the first one this is straightforward, and for the second one
   write
   $\hat \bv = \cB^1 \bv = J_F^{T} (\bv \circ F)$ and 
   $\hat \bw = \tilde \cB^1 \bw = (\det J_F)J_F^{-1} (\bw \circ F)$, 
   that is, 
   \begin{equation} \label{hv_hw}
   \hat \bv 
   = 
   \begin{pmatrix}
    \partial_s F_1 (v_1\circ F)
      + \partial_s F_2 (v_2\circ F)
      \\ 
      \partial_\theta F_1 (v_1\circ F)
        + \partial_\theta F_2 (v_2\circ F)      
   \end{pmatrix}
   \qquad \text{ and } \qquad
   \hat \bw 
   = 
   \begin{pmatrix}
      \partial_\theta F_2 (w_1\circ F)
      - \partial_\theta F_1 (w_2\circ F)
      \\ 
      - \partial_s F_2  (w_1\circ F)
        + \partial_s F_1 (w_2\circ F)      
   \end{pmatrix}
   \end{equation}
   which shows that
   $\hat \bv \cdot \hat \bw = \det J_F ((\bv \cdot \bw) \circ F)$
   holds indeed.
   Using next \eqref{JF_prop}, one also infers from \eqref{hv_hw} that
   $$ 
   \hat \bn \times \hat \bv(0,\theta) 
      = - \hat v_\theta(0,\theta) = 0
      \qquad \text{ and } \qquad 
      \hat \bn \cdot \hat \bw(0,\theta) 
      = - \hat w_s(0,\theta) = 0
   $$
   which proves the second and third equations in \eqref{bc_pb}.
   The first equation is straightforward.
\end{proof}

\subsection{Singular pushforwards of smooth logical functions} 
\label{sec:pf}

We now carefully study the pushforwards of functions which are smooth 
on the logical domain.
Since these pushforwards are not a priori defined at the pole, 
we see the operators \eqref{pf} as 
mappings between measurable functions on the domains 
$\hat \Omega$ and $\Omega$.

\medskip
\begin{prop} \label{prop:cF}
   The pushforward operators satisfy the following properties:
   $$
   \left\{
   \begin{aligned}
      &\text{%
      $\cF^0$ maps continuously
      $L^\infty(\hat \Omega)$ into $L^\infty(\Omega)$
      }
      \\
      &\text{%
      both 
      $\cF^1$ and $\cF^2$ 
      map continuously
      $L^\infty(\hat \Omega)$
      into $L^q(\Omega)$, $1 \le q < 2$}
      \\
      &\text{%
      $\cF^2$ maps continuously
      $L^1(\hat \Omega)$ into $L^1(\Omega)$.}      
   \end{aligned}\right.
   $$
\end{prop}
\smallskip

\begin{proof}
   The first and last statement are straightforward since 
   $\cF^0$ and $\cF^2$ preserve the $L^\infty$ and $L^1$ norms, respectively.
   Next, consider $\hat \bv \in L^\infty(\hat \Omega)$ and compute
   $$
   (\cF^1 \hat \bv)(\bx) = (J_F^{-T} \hat \bv)(s,\theta)
   = \frac{1}{\det J_F(s,\theta)}
   \begin{pmatrix}
    (\partial_\theta F_2 \hat v_s
      - \partial_s F_2  \hat v_\theta)(s,\theta)
      \\ 
      (-\partial_\theta F_1 \hat v_s
        + \partial_s F_1 \hat v_\theta)(s,\theta)
   \end{pmatrix}
   \quad \text{where} \quad \bx = F(s,\theta).
   $$
   In particular, for $1 \le q < 2$ the bounds in \eqref{JF_prop}--\eqref{Os} yield
   $$
   \int_\Omega \abs{\cF^1 \hat \bv}^q \rmd \bx= \int_{\hat \Omega} \det J_F \abs{\cF^1 \hat \bv}^q \rmd(s,\theta)
   \le C \norm{\hat \bv}^q_{L^\infty} \int_{\hat \Omega} (\det J_F)^{1-q} \rmd(s,\theta)
   \le C \norm{\hat \bv}^q_{L^\infty} \int_{\hat \Omega} s^{1-q} \rmd(s,\theta) 
   < \infty
   $$
   since $-1 < 1-q$.
   This shows the continuity of $\cF^1: L^\infty(\hat \Omega) \to L^q(\Omega)$,
   and a similar computation shows that of
   $\cF^2: L^\infty(\hat \Omega) \to L^q(\Omega)$.
\end{proof}

We now specify the commutation properties (or lack thereof) of the pushforward operators. 
\medskip
\begin{prop} \label{prop:cF_diff}
   If $\hat \phi \in W^{1,\infty}(\hat \Omega)$, 
   then $\cF^0 \hat \phi \in W^{1,q}(\Omega)$ for $1 \le q < 2$,
   and the commuting relation holds:
   \begin{equation} \label{com_pf_grad_Lip}
      \bgrad \cF^0 \hat \phi = \cF^1 (\hbgrad \hat \phi)
   \end{equation}
   where $\cF^1 (\hbgrad \hat \phi) \in L^q(\Omega)$ according to Proposition~\ref{prop:cF}.
   If $\hat \bv \in W^{1,\infty}(\hat \Omega)$,
   then its pushforward $\cF^1 \hat \bv$ is in $L^{q}(\Omega)$ for $1 \le q < 2$,
   with a distributional curl of the form
   \begin{equation} \label{com_pf_curl_Lip}
      \curl \cF^1 \hat \bv = \cF^2 (\hcurl \hat \bv) + \alpha \delta_{\bx_0}
      \qquad \text{ with } \qquad \alpha = \int_0^{2\pi} \hat v_\theta(0,\theta) \rmd \theta 
   \end{equation}
   where $\cF^2 (\hcurl \hat \bv) \in L^q(\Omega)$ according to Proposition~\ref{prop:cF}, and 
   $\delta_{\bx_0}$ is the Dirac measure at the pole.   
\end{prop}

\smallskip
This result establishes
that the singular pushforwards {\em do not} commute with the differential 
operators for general smooth logical functions. It also provides us with a 
necessary and sufficient condition for this commutation to hold.
(Note that here we use the standard notation for Sobolev spaces, which also involves 
the letter $W$: this should not be confused with the tensor-product spline spaces.)

\begin{proof}
Since $\hat \phi \in W^{1,\infty}(\hat \Omega)$ is bounded, its pushforward 
$\phi = \cF^0 \hat \phi$ is also bounded, hence in $L^q(\Omega)$.
Its gradient is a distribution which satisfies, 
for all $\bw \in C^1_c(\Omega)$ with pullback 
$\hat \bw = \tilde \cB^1 \bw$, 
$$
   \sprod{\bgrad \phi}{\bw} 
   = -\int_\Omega \phi \div \bw 
   = -\int_{\hat \Omega} \hat \phi \hdiv \hat \bw 
   = \int_{\hat \Omega} \hbgrad \hat \phi \cdot \hat \bw 
      - \int_{\partial \hat \Omega} \hat \phi (\hat \bn \cdot \hat \bw) 
   = \int_{\Omega} \cF^1 (\hbgrad \hat \phi) \cdot \bw.
$$
Here, the second equality follows from the change of variable \eqref{cov}
with the commuting relation \eqref{com_pb_div},
the third one is an integration by part
and in the fourth equality we have used again the change of 
variable \eqref{cov} (note that $\cF^1 (\hbgrad \hat \phi) \in L^1(\Omega)$ according to Proposition~\ref{prop:cF}),
and the fact that the boundary term $\hat\bn \cdot \hat \bw$ 
vanishes on both edges $\{s=L\}$ and $\{s=0\}$:
on the former it follows from the compact support of $\bw$ in the open domain $\Omega$,
and on the latter it follows from \eqref{bc_pb}.

We next consider $\hat \bv \in L^\infty(\hat \Omega)$:
its pushforward $\bv = \cF^1 \hat \bv$ is in $L^q(\Omega)$ 
according again to Proposition~\ref{prop:cF},
and its curl is a distribution that satisfies,
for all $g \in C^1_c(\Omega)$ with pullback 
$\hat g = \tilde \cB^0 g$, 
\begin{equation} \label{curl_pf}
   \begin{aligned}
      \sprod{\curl \bv}{g} 
      &= \int_\Omega \bv \cdot \bcurl g
      = \int_{\hat \Omega} \hat \bv \cdot \hbcurl \hat g
      = \int_{\hat \Omega} \hcurl \hat \bv \, \hat g
         - \int_{\partial \hat \Omega} (\hat \bn \times \hat \bv) \, \hat g
      \\
      &= \int_{\Omega} \cF^2(\hcurl \hat \bv )\, g 
         + g(\bx_0) \int_0^{2\pi} \hat v_\theta(0,\theta)\rmd \theta.
   \end{aligned}
\end{equation}
Here the second equality follows from the change of variable \eqref{cov}
with the commuting relation \eqref{com_pb_bcurl}, 
the third one is an integration by part,
and in the last one we have used again the change of 
variable \eqref{cov} (note that 
$\cF^2(\hcurl \hat \bv ) \in L^1(\Omega)$
thanks to Proposition~\ref{prop:cF}). 
As for the boundary term, we have used that $\hat g = 0$ 
on the edge $\{s=L\}$ (for the same reason as above),
while on $\{s=0\}$ one has
$((\hat \bn \times \hat \bv) \, \hat g)(0, \theta) 
   = -\hat v_\theta(0,\theta) g(\bx_0)$.
This shows the relation \eqref{com_pf_curl_Lip}.
\end{proof}

\section{Commutation properties of the conforming-geometric projections}
\label{sec:commutation}


In Sections~\ref{sec:PV} and \ref{sec:PU} we have defined 
{\em conforming projections} $P^\ell_Z: W^\ell_h \to Z^\ell_h$ 
which map arbitrary tensor-product splines,
a priori singular on the polar domain, 
to splines belonging to the underlying spaces
$Z^\ell = V^\ell$ or $U^\ell$.

In this section we study the commuting properties of 
the composed projections \eqref{Pi_Z}, namely
$$
\Pi^\ell_Z \defeq P^\ell_Z \Pi^\ell_W
$$ 
where we remind that $\Pi^\ell_W$ is the {\em geometric projection} introduced in Section~\ref{sec:geoproj}.

Our objective is to complete the proofs of Theorems~\ref{thm:PV} and \ref{thm:PU}. 
We will first establish some properties of the 
geometric projections:

\medskip
\begin{prop} \label{cd_Pi_W}
   The geometric projection operators $\Pi^\ell_W$ defined in section 
   \ref{sec:geoproj} are well-defined on the domain spaces
   \begin{equation} \label{D_Pi_W}      
      D(\Pi^0_W) = H^1(\Omega) \cap C^0(\overline \Omega),
      \qquad 
      D(\Pi^1_W) = H(\curl;\Omega) \cap C^0(\overline \Omega),
      \qquad
      D(\Pi^2_W) = L^2(\Omega)
   \end{equation}
   and the projections of arbitrary continuous $\phi$ and $\bv$, 
   namely $\Pi^0_W \phi = \sum_{ij} \phi_{ij} T^0_{ij}$ 
   and
   $\Pi^1_W \bv = \sum_{ij} v^s_{ij} \pmb{T}^s_{ij} 
   + \sum_{ij} v^\theta_{ij} \pmb{T}^\theta_{ij}$, 
   satisfy
   \begin{equation} \label{pb_0_DPi}
      \phi_{0j} = \phi_{00} 
      \qquad \text{ and } \qquad 
      v^\theta_{0j} = 0
      \qquad \text{ for all $j$}.
   \end{equation}
   Moreover, the following commutation properties hold
   \begin{equation} \label{com_Pi_W}
   \begin{aligned}
      \Pi^1_W \bgrad \phi &= \bgrad \Pi^0_W \phi
      \\
      \Pi^2_W \curl \bv &= \curl \Pi^1_W \bv  
   \end{aligned}
   \end{equation}
   for all $\phi \in C^1(\overline \Omega)$ and all 
   $\bv \in C^0(\overline \Omega)$ 
   such that $\curl \bv \in L^2(\Omega)$.

\end{prop}
\smallskip

The proof will be given in Section~\ref{sec:proof_cd_Pi_W}.
We will next establish some properties of the 
discrete conforming projections:

\medskip
\begin{prop} \label{cd_PZ}
   For $Z = V$ or $U$, the discrete conforming projections 
   $P^\ell_Z: W^\ell_h \to Z^\ell_h$ 
   satisfy the relations
   \begin{equation} \label{comm_prop_P0P1}
      P^{1}_Z \bgrad \phi_h = \bgrad P^0_Z \phi_h
      \qquad \text{ for all } \phi_h \in W^0_h \text{ such that } 
      \phi_{0j} = \phi_{00} \text{ for all $j$}
   \end{equation}
   and 
   \begin{equation} \label{comm_prop_P1P2}
      P^{2}_Z \curl \bv_h = \curl \hat P^1_Z \bv_h
      \qquad \text{ for all } \bv_h \in W^1_h \text{ such that } 
      v^\theta_{ij} = 0 \text{ for all $j$}.
   \end{equation}
\end{prop}
\smallskip

The proof will be given in Section~\ref{sec:proof_cd_PZ}.
The desired commutation properties of the composed 
conforming-geometric projections follow as a direct corollary:
this completes the proofs of Theorems~\ref{thm:PV} and \ref{thm:PU}.

\medskip

\begin{cor}
   For $Z = V$ or $U$, 
   the commuting relations 
   \begin{equation} \label{com_Pi_Z}
      \begin{aligned}
         \Pi^1_Z \bgrad \phi &= \bgrad \Pi^0_Z \phi
         \\
         \Pi^2_Z \curl \bv &= \curl \Pi^1_Z \bv  
      \end{aligned}
   \end{equation}   
   hold for all $\phi \in C^1(\overline \Omega)$ and all 
   $\bv \in C^0(\overline \Omega)$ 
   such that $\curl \bv \in L^1(\Omega)$.
\end{cor}

\subsection{Properties of the polar geometric projections}
\label{sec:proof_cd_Pi_W} 

In this section we establish Proposition~\ref{cd_Pi_W}.
We first verify (using Proposition~\ref{prop:com_pb} and \ref{prop:pb_0})
that the pullbacks \eqref{pb} satisfy
$\cB^\ell(D^\ell(\Pi^\ell_W)) \subset D^\ell(\hat \Pi^\ell_W)$,
see \eqref{D_hPiW}, so that the operators $\Pi^\ell_W$
are indeed well-defined on the domains \eqref{D_Pi_W}.
By composition with the discrete projections, the projections 
$\Pi^\ell_Z$ are also well-defined.
The relations \eqref{pb_0_DPi} on the coefficients
then follow from the interpolation properties of the 
tensor-product projections $\hat \Pi^0_W$ and $\hat \Pi^1_W$, 
and the fact that the first node along $s$ is $\zeta^s_0 = 0$.

Turning to the commutation properties, we use the definition \eqref{Pi_W}
of the geometric projections and compute
\begin{equation*}
   \begin{aligned}
      \Pi^1_W \bgrad \phi 
      &= \cF^1 \hat \Pi^1_W \cB^1 \bgrad \phi
      = \cF^1 \hat \Pi^1_W \hbgrad \hat \phi
      \\
      &= \cF^1 \hbgrad \hat \Pi^0_W \hat \phi 
      = \bgrad \cF^0 \hat \Pi^0_W \hat \phi 
      = \bgrad \Pi^0_W \phi 
   \end{aligned}
\end{equation*}
where the second equality follows from Proposition~\ref{prop:com_pb}, 
the third one uses the commutation properties of the geometric projections on the 
logical tensor-product grid \eqref{com_hPi_W}, 
and the fourth one follows from Proposition~\ref{prop:cF_diff}, using that 
functions in $\hat W^0$ are always Lipschitz continuous. The last equality is the definition of $\Pi^0_W$.
To show the second relation of \eqref{com_Pi_W}, that is the commutation with the curl, some care is needed since the pushforward operators only commute with 
the curl when the logical functions are normal to the singular edge $\{s=0\}$, 
as seen in \eqref{com_pf_curl_Lip}. 
Here this property holds for $\hat \Pi^1_W \hat \bv$ as a consequence 
of \eqref{pb_0_DPi}, 
so that \eqref{com_pf_curl_Lip} yields indeed 
$\curl \cF^1 \hat \Pi^1_W \hat \bv = \cF^2 \hcurl \hat \Pi^1_W \hat \bv$.
It follows that we may proceed similarly as above and compute
\begin{equation*}
   \begin{aligned}
      \Pi^2_W \curl \bv 
      = \cF^2 \hat \Pi^2_W \cB^2 \curl \bv
      = \cF^2 \hat \Pi^2_W \hcurl \hat \bv
      = \cF^2 \hcurl \hat \Pi^1_W \hat \bv
      = \curl \cF^1 \hat \Pi^1_W \hat \bv 
      = \curl \Pi^1_W \bv
   \end{aligned}
\end{equation*}
where the second equality follows from Proposition~\ref{prop:com_pb},
the third one uses the commutation properties of the 
geometric projections on the logical tensor-product grid \eqref{com_hPi_W}, 
and the fourth one follows from Proposition~\ref{prop:cF_diff} and the observation above.
The last equality is the definition of $\Pi^1_W$.

\subsection{Commutation properties of the conforming projections}
\label{sec:proof_cd_PZ}.

In this section we prove Proposition~\ref{cd_PZ}.
For this purpose we first convert the gradient and curl formulas 
in logical variables, namely \eqref{hgrad_basis}
and \eqref{hcurl_basis}, in physical variables.
We do so by invoking Proposition~\ref{prop:cF_diff} 
which allows us to write that 
the pushforward $\phi_h = \cF^0 \hat \phi_h$
of any $\hat \phi_h \in \hat W^0_h$ satisfies 
\begin{equation} \label{cd_grad_W0h}
\phi_h \in L^{1}(\Omega)
\quad \text{and} \quad
\bgrad \phi_h = \cF^1(\hbgrad \hat \phi_h) \in L^1(\Omega),
\end{equation}
and that the pushforward $\bv_h = \cF^1 \hat \bv_h$ of any $\hat \bv_h \in \hat W^1_h$
with $v^\theta_{0j} = 0$ satisfy
\begin{equation} \label{cd_curl_W0h}
\bv_h \in L^{1}(\Omega)
\quad \text{and} \quad
\curl \bv_h = \cF^2(\hcurl \hat \bv_h) \in L^1(\Omega)
\end{equation}
Thus, we have
\begin{equation}
   \label{grad_T0}
   \bgrad T^0_{ij} = \pmb{T}^s_{(i-1)j} - \pmb{T}^s_{ij} 
   - (\pmb{T}^\theta_{i(j-1)} - \pmb{T}^\theta_{ij})
\end{equation}
(where we have $\pmb{T}^s_{(-1)j}=0$ by convention) and 
\begin{equation}
   \label{curl_T1}
   \curl \pmb{T}^s_{ij} = -(T^2_{i(j-1)} - T^2_{ij})
   \qquad \text{ and } \qquad
   \curl \pmb{T}^\theta_{ij} = T^2_{(i-1)j} - T^2_{ij} 
   \qquad \text{ for } i \ge 1.
\end{equation}
We may now proceed with the proof.

\begin{proof}[Proof of Proposition~\ref{cd_PZ} for $Z=V$]
We begin by observing that if $\phi_h \in W^0_h$ is such that $\phi_{0j} = \phi_{00}$ for all $j$,
then $\phi_h \in V^0_h$ and $\bgrad \phi_h \in W^1_h \cap H(\curl;\Omega) = V^1_h$, so that
the projection properties of $P^\ell_V$ readily yield
$$
\bgrad P^0_V \phi_h = \bgrad \phi_h = P^1_V \bgrad \phi_h
$$
which proves \eqref{comm_prop_P0P1} in the case where $Z=V$.

Turning to the second relation \eqref{comm_prop_P1P2}, 
we use the definitions of $P^1_V$ and $P^2_V$ in \eqref{eq:PV1_def} and \eqref{eq:PV2_def}, as well as the 
relations \eqref{curl_T1} above, to compute:

For $\bv_h = \pmb{T}^s_{0j}$, 
\begin{multline*}
   \curl P^1_V \pmb{T}^s_{0j} 
   = \curl (\pmb{T}_{0j}^s + \pmb{T}_{1(j-1)}^\theta - \pmb{T}_{1j}^\theta)
   \\
   =  -(T^2_{0(j-1)} - T^2_{0j}) + (T^2_{0(j-1)} - T^2_{1(j-1)}) - (T^2_{0j} - T^2_{1j})
   = T^2_{1j} - T^2_{1(j-1)}
\end{multline*}
and 
$$
P^2_V \curl \pmb{T}^s_{0j} = - P^2_V (T^2_{0(j-1)} - T^2_{0j}) = T^2_{1j} - T^2_{1(j-1)} = \curl P^1_V \pmb{T}^s_{0j}. 
$$
For $\bv_h = \pmb{T}^\theta_{1j}$, we have 
$$
\curl P^1_V \pmb{T}^\theta_{1j} = \curl 0 = 0
$$
and
$$
P^2_V \curl \pmb{T}^\theta_{1j} = P^2_V (T^2_{0j} - T^2_{1j}) = T^2_{1j} - T^2_{1j} = 0 = \curl P^1_V \pmb{T}^\theta_{1j}.
$$
Since the other basis functions (namely $\bv_h = \pmb{T}^s_{ij}$ with $i \ge 1$ 
and $\bv_h = \pmb{T}^\theta_{ij}$ with $i \ge 2$)
belong to $V^1_h$, their curl belong to $V^2_h$ and 
hence they satisfy $\curl P^1_h \bv_h = \curl \bv_h = P^2_V \curl \bv_h$.
\end{proof}

\begin{proof}[Proof of Proposition~\ref{cd_PZ} for $Z=U$]
We list the (basis) functions in $W^0_h$ which satisfy the constraint $\phi_{0j} = \phi_{00}$ for all $j$
but do not belong to $U^0_h$, and for each of them 
we use the definitions of $P^0_U$ and $P^1_U$ in \eqref{eq:PU0_def} and \eqref{eq:PU1_def}, 
as well as the relations \eqref{grad_T0} above.

\begin{itemize}
   \item 
   The first one is the linear combination $\phi_h = \sum_j T^0_{0j}$. 
   We note that
   $$
   \bgrad \sum_j T^0_{ij} = \sum_j \big( \pmb{T}^s_{(i-1)j} - \pmb{T}^s_{ij}
   - (\pmb{T}^\theta_{i(j-1)} - \pmb{T}^\theta_{ij})\big)
   = \sum_j (\pmb{T}^s_{(i-1)j} - \pmb{T}^s_{ij}),
   $$
   so that  
   $$
   \bgrad P^0_U \phi_h =  \bgrad \Big(\sum_j (T^0_{0j} + T^0_{1j}) \Big)
   = \sum_j (- \pmb{T}^s_{0j} + (\pmb{T}^s_{0j} - \pmb{T}^s_{1j}))
   = \sum_j (- \pmb{T}^s_{1j})
   $$
   and 
   $$
   P^1_U \bgrad \phi_h = P^1_U \Big(\sum_j (- \pmb{T}^s_{0j})\Big) 
   =
   -\sum_j \pmb{T}_{1j}^s + \pmb{R} 
   $$
   hold with a remainder
   $$
   \pmb{R} = \frac{2}{n_\theta} \sum_{j,k=0}^{n_\theta-1}
   \cos(\theta_k-\theta_j) \big(
   \pmb{T}_{0k}^s - \pmb{T}_{1k}^s
   + \pmb{T}_{1(k-1)}^\theta - \pmb{T}_{1k}^\theta
   \big) = 0.
   $$
   Here we have used that 
   $\sum_j \cos(\theta_k-\theta_j) = \sum_j \cos(\theta_{k-j}) = 0$ 
   which follows from the regular angles and the discrete trigonometric relations \eqref{eq:trigo}.

   \item The second one is $\phi_h = T^0_{1j}$, for which we compute
   $$
   P_U^1 \bgrad T_{1j}^0 
   = P_U^1 \big( \pmb{T}^s_{0j} - \pmb{T}^s_{1j}
   - (\pmb{T}^\theta_{1(j-1)} - \pmb{T}^\theta_{1j})\big)
   = P_U^1 \pmb{T}^s_{0j} - \pmb{T}^s_{1j}
   = \bgrad P_U^0 T^0_{1j}
   $$
   where the last equality follows from \eqref{P1_gradP0}.
\end{itemize}
This shows \eqref{comm_prop_P0P1} in the case where $Z=U$.

Turning to the second relation \eqref{comm_prop_P1P2}, 
we use the definitions of $P^1_U$ and $P^2_U = P^2_V$ in \eqref{eq:PU1_def} and \eqref{eq:PV2_def}, as well as the 
relations \eqref{curl_T1} above. 
Again, we list the (basis) functions $\bv_h \in W^1_h$ which satisfy the constraint $v^\theta_{0j} = 0$ 
but do not belong to $U^1_h$.

\begin{itemize}
   \item The first one is $\bv_h = T^s_{0j}$, for which we compute 
   (using again \eqref{P1_gradP0})
   $$
   \curl P^1_U \pmb{T}^s_{0j} = \curl \big(\pmb{T}^s_{1j} + \bgrad P_U^0 T^0_{1j}\big)
   = \curl \pmb{T}^s_{1j}
   = -(T^2_{1(j-1)} - T^2_{1j})
   $$
   and 
   $$
   P^2_U \curl \pmb{T}^s_{0j} 
   = - P^2_U (T^2_{0(j-1)} - T^2_{0j})
   = - (T^2_{1(j-1)} - T^2_{1j}) = \curl P^1_U \pmb{T}^s_{0j}.
   $$
   
   \item The second one is $\bv_h = \pmb{T}_{1j}^\theta$, on which 
   $P_U^1$ and $P_V^1$ coincide (both vanish). 
   Since $P_U^2 = P_V^2$ by construction,
   the relation 
   $\curl P_U^1 \pmb{T}_{1j}^\theta = P_U^2 \curl \pmb{T}_{1j}^\theta$
   then follows from the case $Z=V$.

\end{itemize}
This shows \eqref{comm_prop_P1P2} in the case where $Z=U$ and completes the proof.

\end{proof}

\section{Numerical illustrations}
\label{sec:num}

In this section we use the polar projections proposed above to solve 
two problems with the broken-FEEC method described in Section~\ref{sec:conga}.


Here the domain $\Omega = \Omega_h$ is an approximated disk 
associated to a spline mapping of the form \eqref{F_spol}.
Specifically, we first consider a disk domain 
\begin{equation}
   \label{disk}
   \Omega^{\rm disk} \defeq \{\bx \in \RR^2 : \norm{\bx}_{2} \le 1\} = F_D(\hat \Omega)
\end{equation}
parametrized by a mapping of the form
\begin{equation} \label{F_D}
   F_D: \begin{pmatrix}
       s \\ \theta 
   \end{pmatrix}
   \to
   \bx_0 
   -
   D
   \begin{pmatrix}
     s^2 \\ 0
  \end{pmatrix}
   + 
   \begin{pmatrix}
       s  \cos \theta \\ s \sin \theta 
   \end{pmatrix}
\quad \text{ with } \quad 
\bx_0 = \begin{pmatrix} D \\ 0 \end{pmatrix},
\end{equation}
where we remind that $\widehat{\Omega} = [0,1] \times (\RR / 2\pi\ZZ)$
and $D \in (-\frac 12, \frac 12)$ corresponds
to a shift of the pole $\bx_0$ relative to the disk center ($\bx_C = \bs{0}$). 
Given a tensor-product spline space $\hat W^0_h = \SS_{p,p}(\hat\Omega)$ of 
degree $p \ge 2$ and ($N_s, N_\theta$) uniform cells 
along the $s$ and $\theta$ directions respectively,
we define a spline mapping $F_{D,h}$ of the form \eqref{F_spol} by 
interpolating $F_D$ on the associated Greville points.
Following \eqref{F_surj}, this defines a computational domain 
$$
\Omega_h = F_{D,h}(\hat \Omega)
$$
where the (multi-) index $h \defeq (p, N_s, N_\theta)$ 
highlights the fact that this domain depends on the discretization parameters.
Note that since open and periodic knots are used along the respective 
$s$ and $\theta$ axes, the dimensions of the univariate spline 
spaces are $n_s = N_s + p$ and $n_\theta = N_\theta$.

\subsection{Poisson problem}
   
We first apply our broken-FEEC method to the Poisson problem 
\begin{equation}\label{eq:PoissonPbm}
\left\{
\begin{array}{rl}
-\Delta \phi = f & \text{in } \Omega^{\rm disk} \\
\phi = 0 & \text{on }\partial\Omega^{\rm disk},
\end{array}
\right.
\end{equation}
with solution and source term defined as
\begin{equation}\label{eq:PoissonSol}
\phi(\bx) = \sin(7 \pi (1 - x^2 - y^2)/2)
\quad \text{ and } \quad 
f = -\Delta \phi.
\end{equation}   

In matrix form, the approximation \eqref{eq:PoissonConga} reads
\begin{equation}\label{eq:PoissonMat}
\big((\mat{G}\mat{P}^0)^T \mat{M}^1 \mat{G} \mat{P}^0 
+ \alpha (\mat{I}^0 - \mat{P}^0)^T \mat{M}^0 (\mat{I}^0 - \mat{P}^0) \big) \arr{\phi} 
= (\mat{P}^0)^T \arr{f}
\end{equation}
where the array 
$\arr{\phi}$ 
contains the spline coefficients of the solution $\phi_h \in W^0_h$, and 
$\arr{f}$ contains the moments of the source, namely
$\arr{\phi}_\mu = \sigma^0_\mu(\phi_h)$ and 
$\arr{f}_\mu = \sprod{f}{\Lambda^0_\mu}$ for $\mu \in \cI^0_h$.


\medskip 
\begin{oss}
   We observe that Equation \eqref{eq:PoissonMat} 
   involves the usual mass matrix $\mat{M}^1$ of $W^1_h$, although it contains
   unbounded values as pointed out in Section~\ref{sec:regmass}.
   The reason for this choice is that it makes actually no difference
   whether one uses $\mat{M}^1$ or the regularized (bounded and invertible)
   matrix $\tmat{M}^1$ in \eqref{eq:PoissonMat}, 
   indeed one has
   $$
   (\mat{G}\mat{P}^0)^T \tmat{M}^1 \mat{G} \mat{P}^0 
   = (\mat{G}\mat{P}^0)^T \mat{M}^1 \mat{G} \mat{P}^0.
   $$
   where we remind that $\tmat{M}^1$ is defined in \eqref{tM}. 
   This follows from the relations
   $$
   \mat{G}\mat{P}^0 = (\mat{G}\mat{P}^0)\mat{P}^0 = \mat{P}^1\mat{G}\mat{P}^0
   \quad \text{ and } \quad 
   (\mat{P}^1)^T \tmat{M}^1 \mat{P}^1 = (\mat{P}^1)^T \mat{M}^1 \mat{P}^1,
   $$
   where the first equality uses the projection and commutation properties of $\mat{P}^0$, while the second one uses \eqref{tM}, and that $\mat{P}^1$ is a projection.
   Note that this observation can also be done at the level of 
   Equation \eqref{eq:PoissonConga}: since $\bgrad P^0_Z$
   always maps into an $L^2$ subspace of $W^1_h$, 
   the usual $L^2$ product may be used for the stiffness matrix.
\end{oss}
   
\medskip
In Figure~\ref{fig:PoissonSol} we plot the approximate solutions $\phi_h$
obtained by our $C^1$ broken-FEEC method, and the associated 
errors $\phi-\phi_h$ on their respective domains $\Omega_h$, 
for a fixed degree $p = 2$ and increasing resolutions.
We note that our assumption \eqref{ntheta} on the number of angular knots is not satisfied for the coarser case 
on the left (where $n_\theta = 3$), in particular several properties of our conforming projections may not hold for 
this coarse spline grid, nevertheless our scheme still computes a approximate solution that can be considered as reasonable
given the low resolution. 
More importantly, we observe that both the domains and the solutions seem to converge towards the exact ones as the resolution increases.
\begin{figure}[h]
   \centering
   \subfloat{
   \includegraphics[width = .28\textwidth]{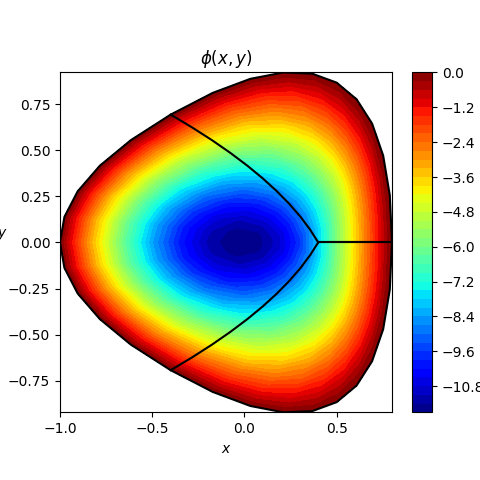}}
   \subfloat{
   \includegraphics[width = .28\textwidth]{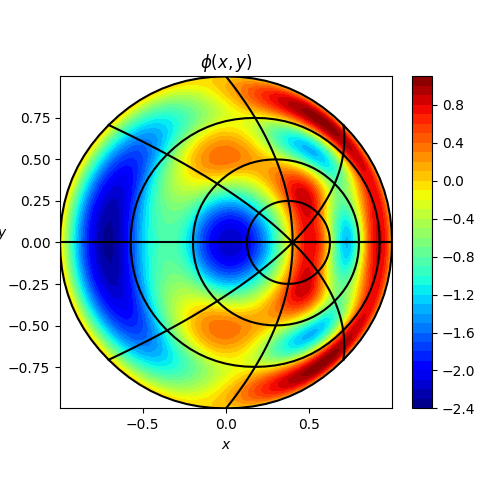}}
   \subfloat{
   \includegraphics[width = .28\textwidth]{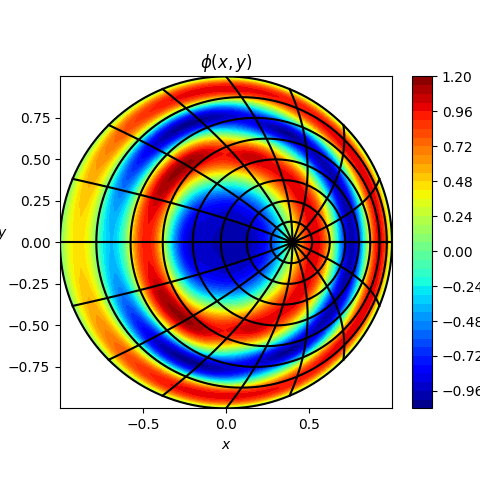}}
   \\[-10pt]
   \subfloat{
   \includegraphics[width = .28\textwidth]{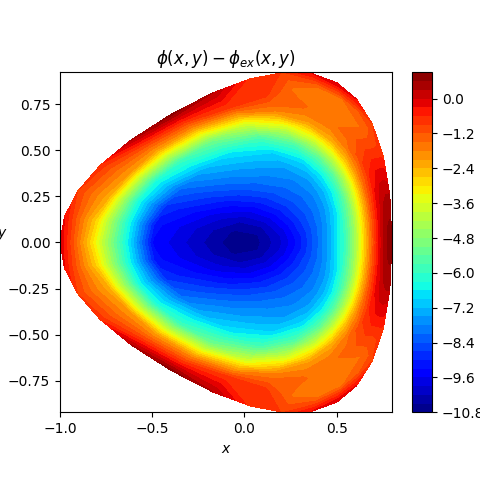}}
   \subfloat{
   \includegraphics[width = .28\textwidth]{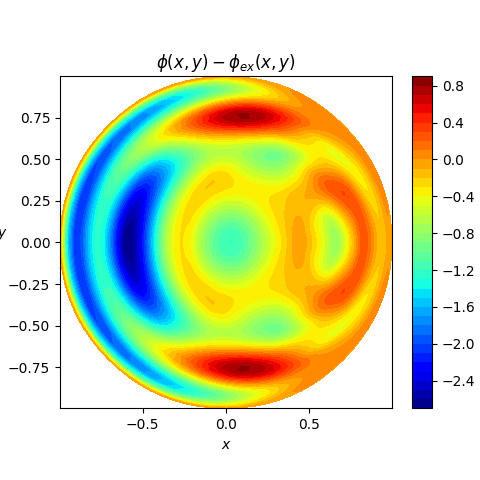}}
   \subfloat{
   \includegraphics[width = .28\textwidth]{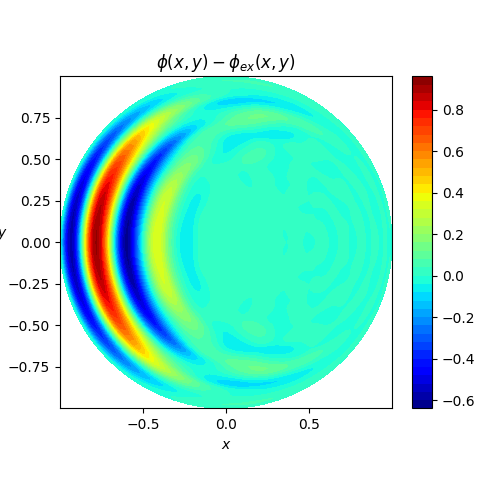}}
   \caption{Poisson problem: approximate solutions $\phi_h$ (top) and 
   errors $\phi_h-\phi$ (bottom) on the discrete domains $\Omega_h$,
   using splines of degree $p=2$ and logical grids with $(N_s, N_\theta) = (1,3)$ 
   cells on the left plots, 
   $(4,8)$ on the middle ones and $(8,16)$ on the right ones.
   In each case the spline mapping $F_{D,h}$ defining $\Omega_h$ uses the same 
   logical grid as the discrete potential $\phi_h$, which results in a visible improvement 
   of the domain itself when the grid is refined.
   The corresponding cells on the mapped domain are drawn 
   with solid lines on the top plots, also showing the pole 
   $\bx_0 = (D, 0)$.
   }\label{fig:PoissonSol}
\end{figure} 
   
To further assess the accuracy of the approximations 
we next perform a numerical convergence analysis,
where we compare three different methods:
\begin{itemize}
   \item the first one is a conforming discretization, where we 
   approximate the Poisson problem by a standard Galerkin projection 
   in the $C^1$ space $U^0_h$, using the polar spline basis
   introduced in \cite{deepesh1}

   \item the second one is a $C^0$ broken-FEEC discretization,
   where the problem is solved in the non-conforming 
   space $W^0_h$ and the $C^0$ projection $P^0_V$ is used 
   to handle the polar singularity

   \item the third one is a $C^1$ broken-FEEC discretization,
   where the problem is solved again in the non-conforming 
   space $W^0_h$, but now using the $C^1$ projection $P^0_U$ 
   to handle the polar singularity.
\end{itemize}

In each case the matrix systems are solved using the conjugate gradient method with a tolerance on the residual error of \texttt{1e-12}. 
For each method we measure the relative errors in $L^2$ and $H^1$ norms,
\begin{equation} \label{err_phi}
   \mathcal{E}_{h, L^2} = \frac{\norm{\Pi^0_W \phi - \phi_h}_{L^2(\Omega_h)}}%
   {\norm{\Pi^0_W \phi}_{L^2(\Omega_h)}},
\qquad 
   \mathcal{E}_{h, H^1} = \frac{\norm{\Pi^0_W \phi - \phi_h}_{H^1(\Omega_h)}}%
   {\norm{\Pi^0_W \phi}_{H^1(\Omega_h)}},
\end{equation}
for various degrees and grid resolutions. 
Here we remind that $\Pi^0_W$ is the spline interpolation 
operator associated with the space $W^0_h$,
so that $\Pi^0_W \phi$ may be considered a discrete 
reference solution on the computational domain $\Omega_h$.

In Figure~\ref{fig:PoissonErr} we plot the 
$L^2$ and $H^1$ errors for the different discretization methods 
against increasing mesh resolutions, for several degrees $p = 2, \dots 5$. 
For both norms we find that the three methods yield the same errors:
for the conforming and $C^1$ broken-FEEC methods this is expected 
since both solutions should coincide \cite{guclu_broken_2023}, and for the 
$C^0$ broken-FEEC method (a priori more accurate since the solution is computed in a 
larger space) we see that the additional 
smoothness constraint on the pole does not degrade the accuracy, which is not
surprising here given the high smoothness of the solution.
In Table~\ref{tab:PoissonRate} we finally show the convergence rates measured 
from the errors on the two finest grids: they are found to be close or higher 
than the standard optimal rates of $p + 1$ for the $L^2$ norm and $p$ for the $H^1$ norm.


\begin{figure}[!htbp]
\centering
\subfloat[relative $L^2$ errors]{
\includegraphics[width = .45\textwidth]{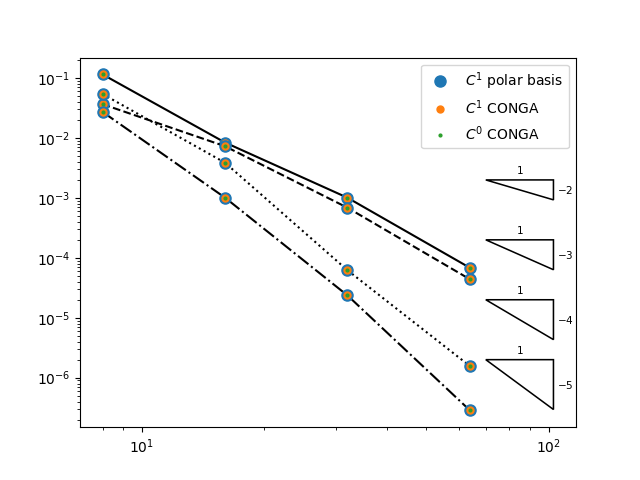}
}\quad
\subfloat[relative $H^1$ errors]{
\includegraphics[width = .45\textwidth]{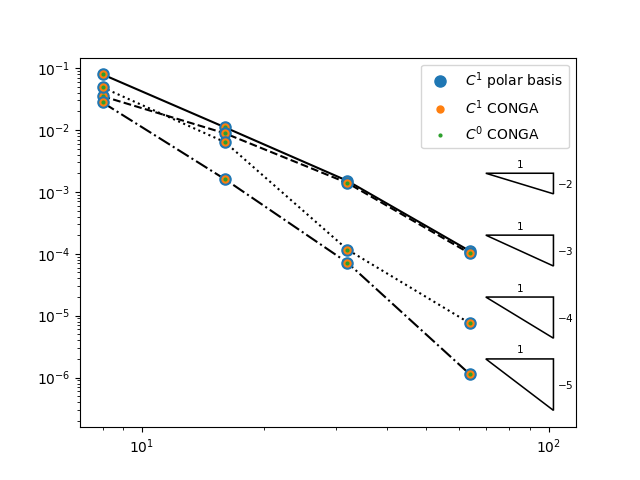}
}
\caption{Poisson problem: relative errors \eqref{err_phi} 
in $L^2$ norm (left) and $H^1$ norm (right) 
are plotted for the three different methods described in the text 
($C^1$ conforming, $C^0$ broken-FEEC and $C^1$ broken-FEEC)
as a function of $N_s$ the number of cells along $s$
(with $N_\theta = 2N_s$), represented in abscissa.
For each method different spline degrees are used, indicated
by different line styles, namely $p = 2, 3, 4, 5$ indicated by 
solid, dashed, dotted and dash-dotted lines respectively. 
}\label{fig:PoissonErr}
\end{figure}

\begin{table} [t] 
   \centering
   \begin{tabular}{|l|l|l|l|l|}
      \hline
   degree $p$
   & 2 
      & 3
         & 4 
            & 5
   \\ 
   \hline 
   $L^2$ rate
   & 3.89 
      & 3.97 
         & 5.32
            & 6.37
   \\ 
   \hline 
   $H^1$ rate
   & 3.77
      & 3.78
         & 3.95
            & 5.99
             \\
      \hline 
   \end{tabular}
   \caption{$L^2$ and $H^1$ convergence rates measured from the two finer grids in the error curves plotted in Figure~\ref{fig:PoissonErr}.}
   \label{tab:PoissonRate}
\end{table}

\subsection{Maxwell's equations}

We next apply our approach to the time dependent 
Maxwell equation \eqref{eq:tMaxwell}.
Associated with a leap-frog time scheme, the broken-FEEC discretization 
\eqref{eq:tMaxwell_h} with regularized mass matrices \eqref{tM} reads
\begin{equation}\label{eq:tMaxwell_hn}
   \left\{
   \begin{array}{rl}
      \arr{B}^{n + \frac 12} &= \arr{B}^n 
         - \frac{\Delta t}{2} \mat{C} \mat{P}^1\arr{E}^n
         \\
      \tmat{M}^1 \arr{E}^{n+1} &= \tmat{M}^1 \arr{E}^n 
         + \Delta t  \big( (\mat{C} \mat{P}^1)^T \tmat{M}^2 \arr{B}^{n+\frac12} - (\mat{P}^1)^T \arr{J}^{n+\frac12} \big) 
         \\
      \arr{B}^{n+1} &= \arr{B}^{n+\frac12} 
         - \frac{\Delta t}{2} \mat{C} \mat{P}^1 \arr{E}^{n+1}
   \end{array}
   \right.
\end{equation}
with a current source array defined as
$
\arr{J}^{n+\frac12}_\mu = \frac{1}{\Delta t} \int_{t^n}^{t^{n+1}}
\sprod{\bJ(t)}{\Lambda^1_\mu} \rmd t
$,
see \cite{campos_compatible_2017}.

To assess the qualitative properties of our polar broken-FEEC scheme, 
we consider a circular wave propagating from the center $\bx_C = (0,0)$
through the pole $\bx_0 = (D, 0)$. 
The initial condition is a Gaussian pulse centered at the disk origin
\begin{equation}
   \bE(t=0, \bx) = 
   \begin{pmatrix}
      y \\ -x 
   \end{pmatrix}
   \exp\Big( -\frac{\norm{\bx}^2_2}{2 \sigma^2}\Big),
   \quad 
   B(t=0, \bx) = \curl \bE^0(\bx)
   \qquad \text{for $\bx = (x,y) \in \RR^2$},
\end{equation}
with $\sigma = 0.1$.
In Figure~\ref{fig:wave} we plot the profiles of the (scalar)
transverse magnetic field $B_z$ at different times, for
increasing resolutions:
while spurious oscillations can be seen around the pole on the coarsest mesh,
they quickly disappear as the resolution is increased.

\begin{figure}[!htbp]
   \centering
   \subfloat{
   \includegraphics[width = .28\textwidth]{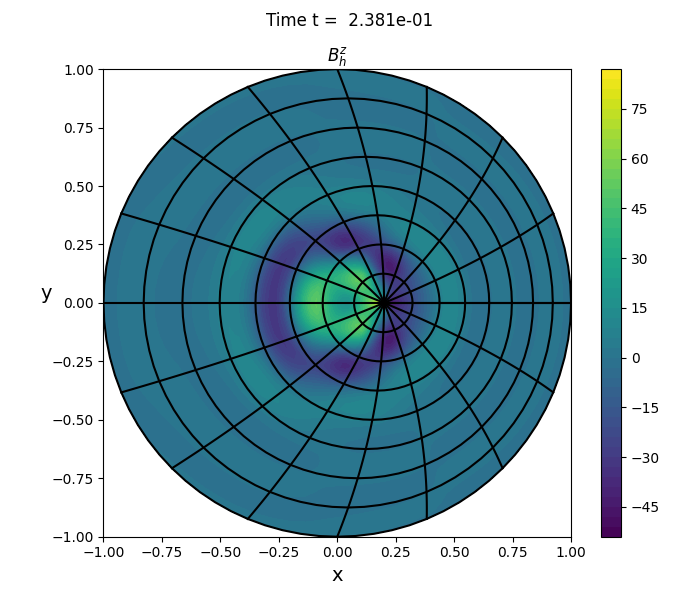}
   }\quad
   \subfloat{
   \includegraphics[width = .28\textwidth]{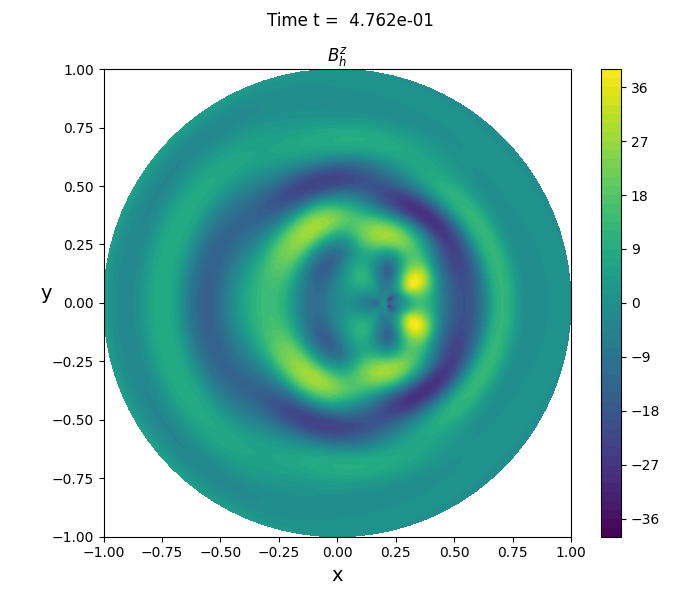}
   }\quad
   \subfloat{
   \includegraphics[width = .28\textwidth]{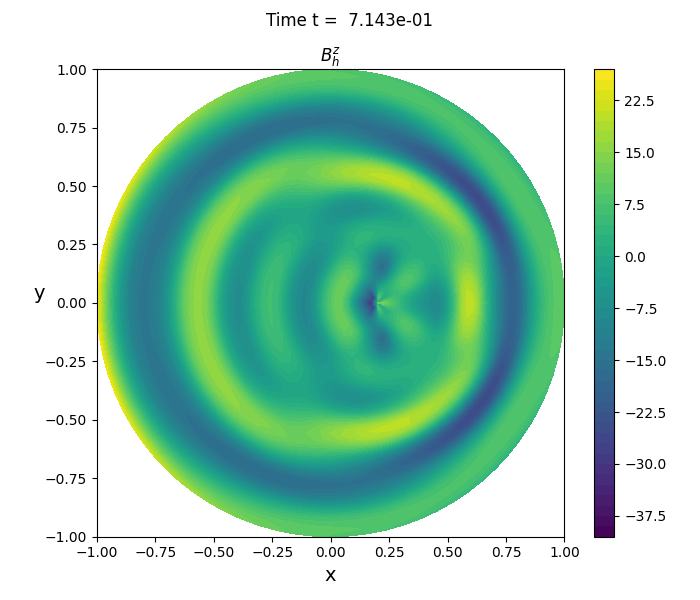}
   }
   \\[-7pt]
   \subfloat{
   \includegraphics[width = .28\textwidth]{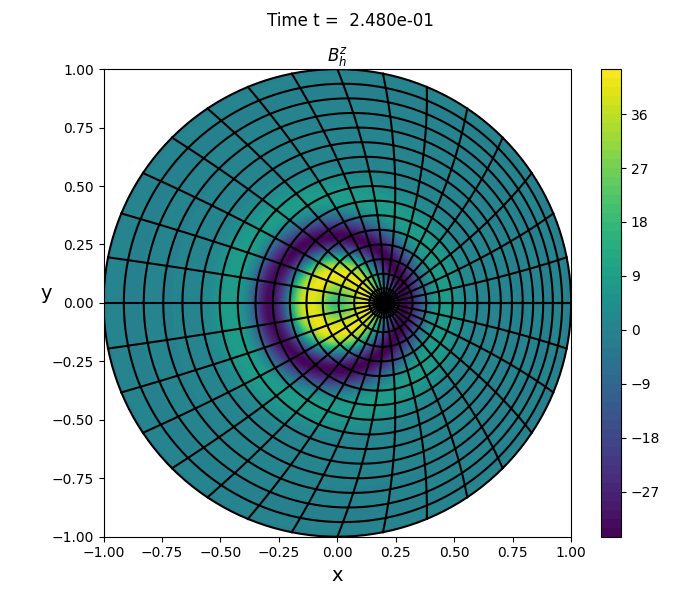}
   }\quad
   \subfloat{
   \includegraphics[width = .28\textwidth]{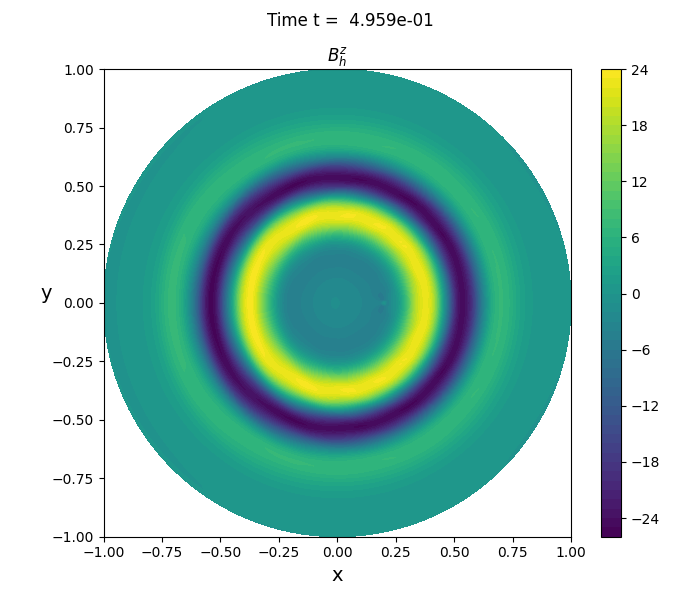}
   }\quad
   \subfloat{
   \includegraphics[width = .28\textwidth]{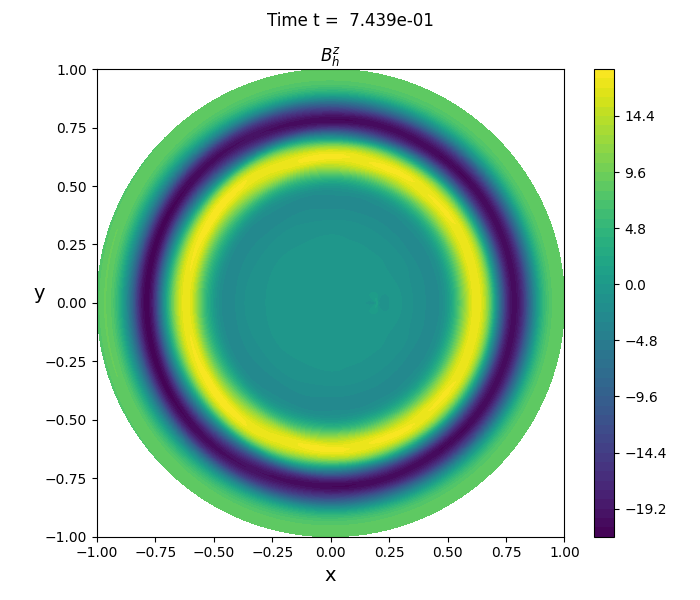}
   }
   \\[-7pt]
   \subfloat{
   \includegraphics[width = .28\textwidth]{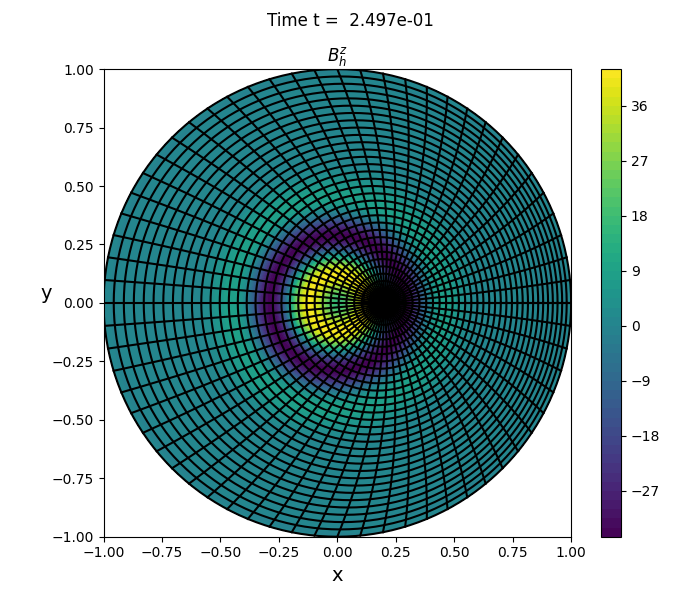}
   }\quad
   \subfloat{
   \includegraphics[width = .28\textwidth]{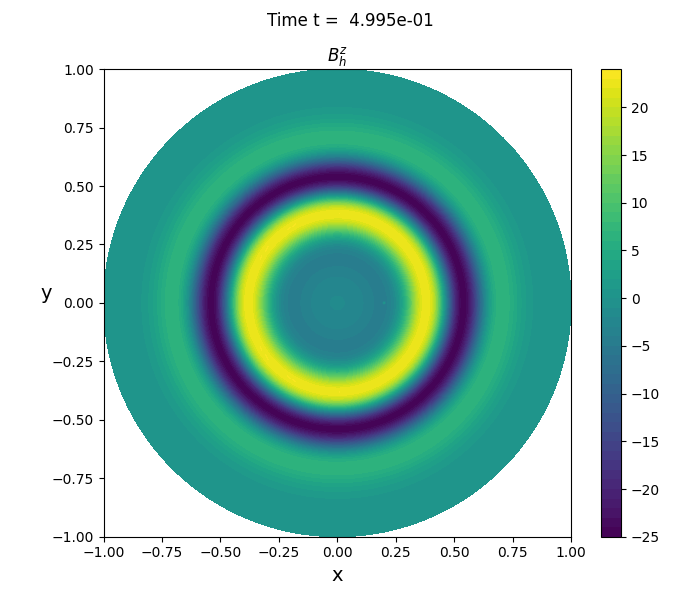}
   }\quad
   \subfloat{
   \includegraphics[width = .28\textwidth]{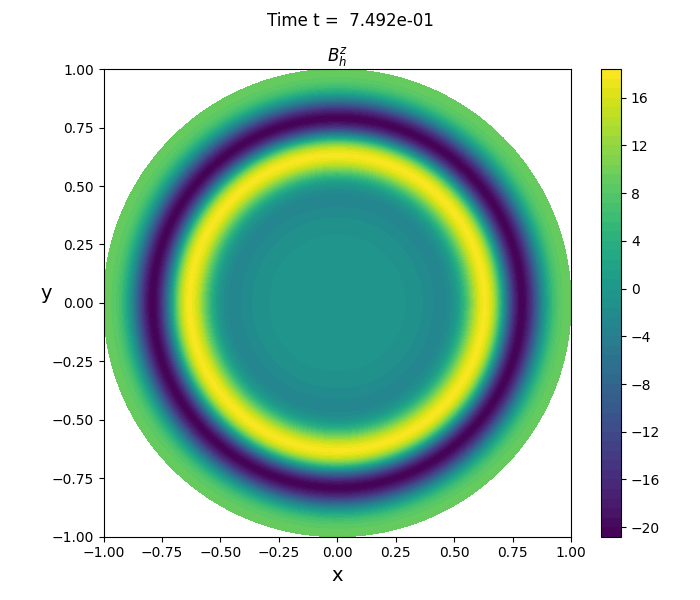}
   }   
   \caption{Maxwell equations: propagation of a circular wave at three successive times 
   ($t \approx 2.5, 5$ and $7.5$, from left to right) computed with the $C^1$ broken-FEEC
   scheme \eqref{eq:tMaxwell_hn}. Three grids (represented on the left plots) are used,
   with $N_s = 8$, $16$ and $32$ cells along $s$, and $N_\theta = 2N_s$ along the periodic variable 
   (from top to bottom).}\label{fig:wave}
\end{figure}

Finally we assess the quantitative convergence of 
our broken-FEEC Maxwell solver with a time-harmonic solution 
of Bessel-Fourier type, with mode $(n,m) \in \NN^2$.
Let us remind how this solution is obtained.
Using the (standard) polar parametrisation 
$F_{D=0}: (s,\theta) \mapsto (x,y) = (s\cos\theta, s\sin\theta)$
of the disk domain \eqref{disk},
the complex solution can be expressed as the pushforward 
\begin{equation} \label{B_E_pf}
   B = \cF^2 \hB, \qquad \bE = \cF^1 \hbE
\end{equation}
of logical fields of the form
\begin{equation} \label{B_E_log}
\hB(t,s,\theta) = s e^{i\omega t} J_n(ks) \cos(n\theta),
\qquad
\hbE(t,s,\theta) = \begin{pmatrix}
   E_s(t,s,\theta)
   \\
   E_\theta(t,s,\theta)
\end{pmatrix}
= i e^{i\omega t} \begin{pmatrix}
   \frac{n}{ks} J_n(ks) \sin(n\theta)
   \\
   s J'_n(ks) \cos(n\theta)
\end{pmatrix}
\end{equation}
where $J_n$ is the Bessel function of the first kind 
and $k = \omega = j'_{n,m}$ is the $m$-th root
of its derivative $J'_n$.

Here, one easily verifies that the logical fields satisfy 
$\partial_t \hB + \hcurl \hbE = 0$, so that Faraday's equation follows
from the commuting properties of the pushforward operators, namely
\eqref{com_pf_curl_Lip} (note that
$E_\theta(s=0) = 0$). To verify that Ampère's equation holds as well,
one can use the explicit forms of the pushforward operators
\eqref{pf} 
to rewrite \eqref{B_E_pf}--\eqref{B_E_log} as
\begin{equation*} \label{B_E_ex}
   \begin{cases}
      B(t,x,y) = \frac{1}{s} \hB(t,s,\theta) = e^{i\omega t} J_n(ks) \cos(n\theta)
   \\[5pt]   
   \bE(t,x,y) 
   = \begin{pmatrix}
      \cos \theta & -\frac 1s \sin \theta
      \\
      \sin \theta & \frac 1s \cos \theta
   \end{pmatrix} \hbE(t,s,\theta) 
   = i e^{i\omega t}  \begin{pmatrix}
      \frac{n}{ks}  J_n(ks) \sin(n\theta)  \cos \theta 
      - J'_n(ks) \cos(n\theta) \sin \theta
      \\
      \frac{n}{ks}  J_n(ks) \sin(n\theta)  \sin \theta 
      + J'_n(ks) \cos(n\theta) \cos \theta
   \end{pmatrix},
   \end{cases}
\end{equation*} 
and the fact that $J_n$ is the (bounded) solution of the differential equation
\begin{equation}
   \label{bessel}
   x^2 J''_n(x) + x J'_n(x) + (x^2-n^2)J_n(x) = 0.
\end{equation}
We remind that close to $s \to 0$ we have 
$J_n(ks) \sim \big(\frac{ks}{2}\big)^n / \Gamma(n+1)$,
see \cite{abramowitz_handbook_1964}, so that 
$\bE$ has no singularity at the pole. 
The homogeneous boundary condition $\bn \times \bE = 0$
is easily verified using $E_\theta(s=1) = 0$ which follows 
from the fact that $J'_n(k) = 0$.

In our numerical convergence study we consider the real part of the 
above Bessel-Fourier solution with mode number $(m, n) = (2, 3)$,
on a time range $[0,T]$ with $T = 0.1$.
The initial solutions are represented in Figure~\ref{fig:bessel}, 
and in Figure~\ref{fig:BesselErr} we plot the errors of our
broken-FEEC scheme \eqref{eq:tMaxwell_hn} at $t=T$, 
with $C^0$ and $C^1$ projection operators, i.e. using
$P^1 = P^1_V$ and $P^1_U$ respectively.
Note that in our experiments a Suzuki-Yoshida \cite{suzuki_fractal_1990,yoshida_construction_1990} time scheme of order 4
is obtained by composing the Strang steps as described in 
\cite[Sec.~5.2]{kraus_gempic_2017}.

Again, we observe that the $C^0$ and $C^1$ schemes 
yield virtually the same errors, except in a couple of 
cases for which we do not have a clear diagnostic.
As for the convergence rates (measured again using the finest grids,
and displayed in Table~\ref{tab:BesselRate}), they are close 
to $p$ which may be considered as optimal given that the logical 
approximation spaces $\hat W^1_h$ and $\hat W^2_h$
consist of B-splines of mixed degree $p$ and $p-1$, see \eqref{dR_W}.

\begin{table}[!htbp]
   \centering
   \begin{tabular}{|l|l|l|l|l|}
   \hline
   degree $p$
   & 2 
      & 3
         & 4 
            & 5
   \\ 
   \hline
   $\bE$ error rate
   & 2.02 
      & 3.05
         & 4.08
            & 5.14
   \\ 
   \hline
   $B$ error rate
   & 2.05 
      & 3.16 
         & 4.20 
            & 4.99
             \\ \hline
   \end{tabular}
   \caption{Convergence rates for the errors in $\bE$ and $B$
   as computed with the $C^1$ broken-FEEC scheme \eqref{eq:tMaxwell_hn}, 
   measured from the two finer grids in the error curves plotted in Figure~\ref{fig:BesselErr}.}
   \label{tab:BesselRate}
\end{table}

Overall these numerical experiments  
confirm the ability of our broken-FEEC approach to approximate
problems on polar domains.

\begin{figure}
   \centering
   \subfloat{
   \includegraphics[width = .28\textwidth]{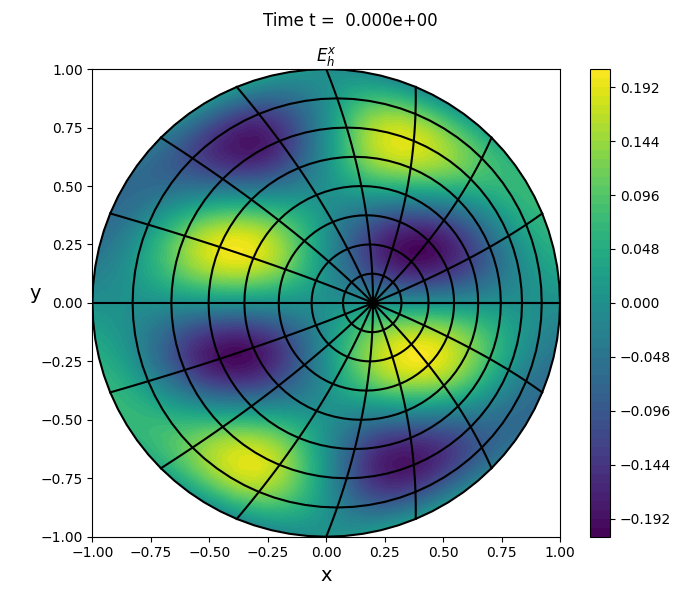}
   }\quad
   \subfloat{
      \includegraphics[width = .28\textwidth]{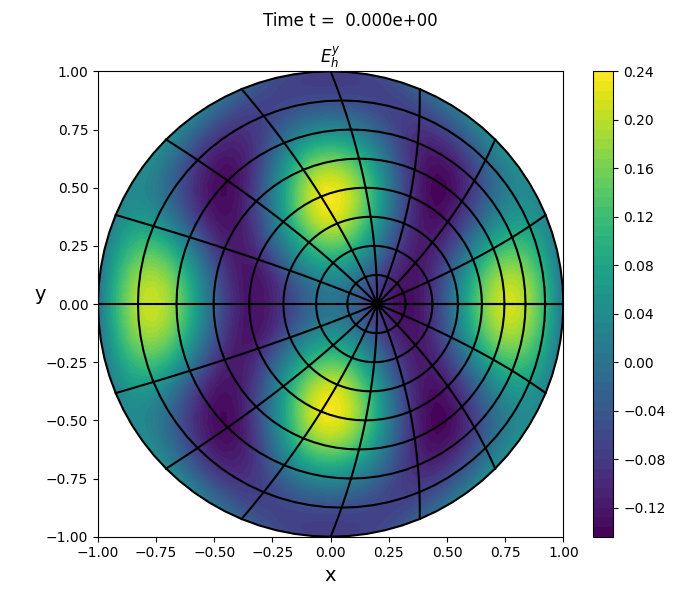}
      }\quad
   \subfloat{
      \includegraphics[width = .28\textwidth]{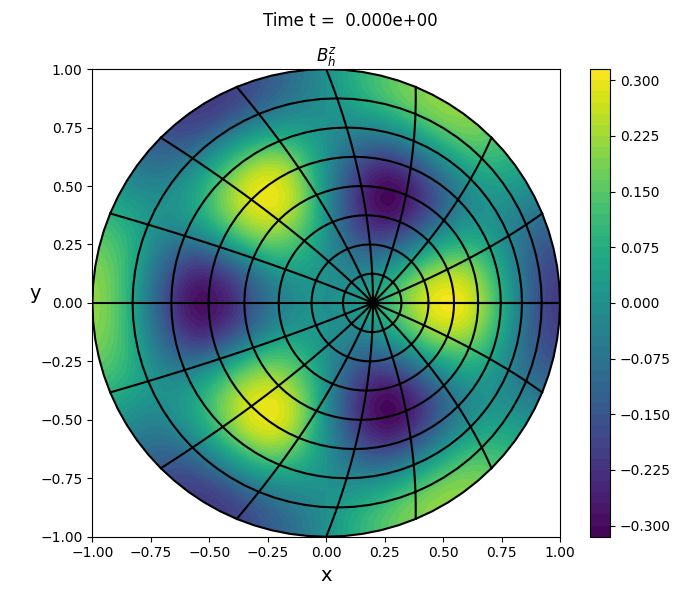}
   }
   \caption{Initial Bessel-Fourier solution 
   (from left to right: $E_x$, $E_y$, $B$) 
   with mode number $(m, n) = (2, 3)$
   used for the convergence study of Maxwell's equations.}
   \label{fig:bessel}
\end{figure}

\begin{figure}
   \centering
   \subfloat[]{
   \includegraphics[width = .45\textwidth]{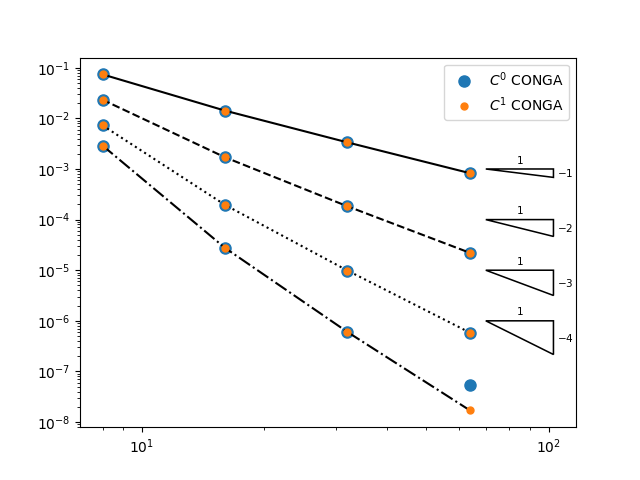}
   }\quad
   \subfloat[]{
   \includegraphics[width = .45\textwidth]{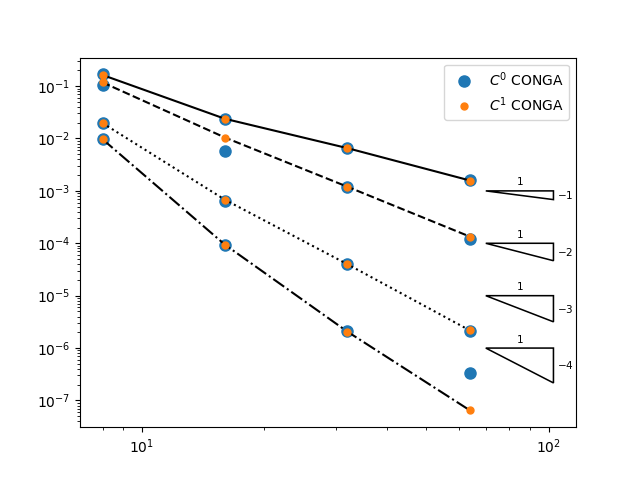}
   }
   \caption{Maxwell's equations: relative $L^2$ errors
   for the $\bE$ (left) and $B$ fields (right) 
   obtained using the polar broken-FEEC scheme \eqref{eq:tMaxwell_hn} 
   with $C^0$ and $C^1$ projection operators
   $P^1 = P^1_V$ and $P^1_U$.
   As in Figure~\ref{fig:PoissonErr}, errors are plotted 
   versus $N_s$ the number of cells along $s$
   (with $N_\theta = 2N_s$) 
   and different spline degrees ($p = 2, 3, 4, 5$) are used, 
   indicated by different line styles
   (solid, dashed, dotted and dash-dotted lines respectively). 
   }\label{fig:BesselErr}
\end{figure}

\section*{Acknowledgments}

The authors would like to thank Stefan Possanner and Roman Hatzky for insightful discussions.

Y. Güçlü and F. Patrizi are members of the \textit{Gruppo Nazionale per il Calcolo Scientifico} of the Italian \textit{Istituto Nazionale di Alta Matematica “Francesco Severi”} (GNCS-INdAM). F. Patrizi acknowledges the MUR Excellence Department Project MatMod@TOV awarded to the Department of Mathematics, University of Rome Tor Vergata, CUP E83C23000330006.

This work has also partly been carried out within the framework
of the EUROfusion Consortium, funded by the European Union via the Euratom Research and Training Programme (Grant Agreement No 101052200 - EUROfusion). Views and opinions expressed are however those of the author(s) only and do not necessarily reflect those of the European Union or the European Commission. Neither the European Union nor the European Commission can be held responsible for them.

\appendix
\section{Reminder on B- and M-splines}

The splines $B_i(s)$ and $\per{B}_j(\theta)$ involved in this article 
correspond to the normalized B-splines, namely
$B_i(s) = (s_{i+m}-s_i)Q^m_i(s)$ with $m = p+1$
in Schumaker's notation \cite{schumaker}.
We further denote
$M_i(s) \defeq (m-1) Q^{m-1}_{i+1}(s)$ for $0 \le i < n_s-1$, 
yielding the derivative formula
\begin{equation} \label{Bprime}
    B_i' = M_{i-1} - M_{i} \quad \text{ for } \quad 0 \le i < n_s
\end{equation}
where we have set $ M_{-1} \defeq M_{n_s-1} \defeq 0$ for convenience.
We remind that B-splines satisfy a partition of unity property, namely
\begin{equation} \label{pup}
    \sum_{i=0}^{n_s-1}B_i(s) = 1, \qquad
    \sum_{j=0}^{n_\theta-1}\per{B}_j(\theta) = 1
\end{equation}
hold for all $s$ and $\theta$, and that those defined 
from open knot sequences are interpolatory at the end points.
In particular, we have 
\begin{equation} \label{BM_0}
  B_i(0) = M_i(0) = 0 \text{ for } i > 0, 
  \quad
  B_0(0) = 1, \quad 
\end{equation}
and 
\begin{equation} \label{BM_0p}
  B'_i(0) = 0  \text{ for } i > 1, 
  \quad 
  B_1'(0) = M_0(0) = \frac{p}{s_{p+1}}.
\end{equation}
A standard relation follows from \eqref{Bprime} and \eqref{BM_0}--\eqref{BM_0p}, 
namely
\begin{equation}   
\label{int_M}
\int_0^L M_i(s) \rmd s = 1 \qquad \text{ for } \quad 0 \le i < n_s - 1.
\end{equation}
For the periodic splines along $\theta$ we define similarly 
$\per{M}_j(\theta) \defeq (m-1) \per{Q}^{m-1}_{j+1}(\theta)$, yielding the same derivative formula
than aboove,
\begin{equation} \label{perBprime}
    \per{B}_j' = \per{M}_{j-1} - \per{M}_{j} \quad \text{ for } \quad 0 \le j < n_s
\end{equation}
where we now observe that every spline $\per{M}_{j}$ 
is defined by periodicity.
With regular knots the periodic M-splines also sum to a constant value,
namely
\begin{equation} \label{pup-M}
    \sum_{j=0}^{n_\theta-1}\per{M}_j(\theta) = \frac{1}{\Delta \theta}
    \quad 
    \text{ with } ~~ \Delta \theta = \frac{2\pi}{n_\theta}.
\end{equation}
Note that with this definition the splines $\per{B}_j(\theta)$ and 
$\per{M}_j(\theta)$ are supported 
in the respective intervals $[\theta_j, \theta_{j+p+1}]$ 
and $[\theta_{j+1}, \theta_{j+p+1}]$, in particular we have
\begin{equation} \label{perBM}
\per{B}_j \per{M}_{j+k} = 0 \quad \text{ for } \quad k \notin \{-p, \dots p-1\}.
\end{equation}

We further list some elementary relations satisfied by the regular angles 
$\theta_j = \frac{2\pi j}{n_\theta}$, which hold as long as $n_\theta = 4n'$ 
for a positive integer $n' > 0$, see Assumption~\ref{ass:ntheta}:
\begin{equation}\label{eq:trigo}
\begin{array}{c}
\displaystyle\sum_{j=0}^{n_\theta-1} \cos(2\theta_j) 
   = \sum_{j=0}^{n_\theta-1} \cos(\theta_j) 
   = \sum_{j=0}^{n_\theta-1} \sin(2\theta_j) 
   = \sum_{j=0}^{n_\theta-1} \sin(\theta_j) = 0
    \\\\
\displaystyle\sum_{j=0}^{n_\theta-1} 2\cos^2\theta_j 
   = \sum_{j=0}^{n_\theta-1} (1+\cos(2\theta_j)) 
   = \sum_{j=0}^{n_\theta-1} 2\sin^2\theta_j 
   = \sum_j (1-\cos(2\theta_j)) 
   = n_\theta
    \\\\
\displaystyle\sum_{j=0}^{n_\theta-1} 2\cos\theta_j\cos(\theta_j - \theta_k) = n_\theta \cos\theta_k
   , \qquad
    \sum_{j=0}^{n_\theta-1} 2\sin\theta_j\cos(\theta_j - \theta_k) = n_\theta \sin\theta_k~.
\end{array}
\end{equation}

\details{Spline formulas: 
we denote 
$B_i = N^{m}_i = (s_{i+m}-s_i)Q^m_i$ the normalized B-spline \cite{schumaker} of order $m = p+1 \le 2$,
so that 
$$
\sum_{i=0}^{n_s-1} B_i = 1
$$
Derivative formula:
$$
(Q^m_i)' = \frac{m-1}{s_{i+m}-s_{i}} (Q^{(m-1)}_i - Q^{(m-1)}_{i+1})
\quad \iff \quad
B_i' = \frac{p}{s_{i+p}-s_{i}} B^{(p-1)}_{i} - \frac{p}{s_{i+1+p}-s_{i+1}} B^{(p-1)}_{i+1}.
$$
We further denote $M_i(s) = \frac{p}{s_{i+1+p}-s_{i+1}} B^{(p-1)}_{i+1}(s)$ for $0 \le i < n_s-1$
(in Schumaker's notation this yields $M_i = (m-1) Q^{m-1}_{i+1}$), 
and also set $M_{-1} = M_{n_s-1} = 0$, so that 
$$
B_i' = M_{i-1} - M_{i} \quad \text{ for } \quad 0 \le i < n_s.
$$
We take $p \ge 1$. The interpolation properties of B-splines with open knot sequence yield
$$
B_0(0) = 1, \quad M_0(0) = \frac{p}{s_{p+1}}
\quad \text{ and } \quad
B_i(0) = M_i(0) = 0 \text{ for } i > 0.
$$
}

\bibliography{biblio}
\end{document}